%% file: sloan_paper_rev1ArXiV.tex
\documentclass{svmult}

\input{macros}

\newcommand{\FTDFigDirectory}{.} 

\usepackage{scalerel,stackengine}
\stackMath
\newcommand\widecheck[1]{%
\savestack{\tmpbox}{\stretchto{%
  \scaleto{%
    \scalerel*[\widthof{\ensuremath{#1}}]{\kern-.6pt\bigwedge\kern-.6pt}%
    {\rule[-\textheight/2]{1ex}{\textheight}}
  }{\textheight}%
}{0.5ex}}%
\stackon[1pt]{#1}{\scalebox{-1}{\tmpbox}}%
}

\def\abs#1{\ensuremath{\left \lvert #1 \right \rvert}}

\begin{document}
	\allowdisplaybreaks

\title*{Adaptive Quasi-Monte Carlo Methods for Cubature}

\author{Fred J. Hickernell \and Llu\'is Antoni Jim\'enez Rugama  \and Da Li}

\institute{
	Fred J. Hickernell (\Letter), Llu\'is Antoni Jim\'enez Rugama \and Da Li 
	\at Illinois Institute of Technology, RE 208, 10 W. 32$^{\text{nd}}$ St., Chicago, USA \\
	\email{hickernell@iit.edu; ljimene1@hawk.iit.edu; dli37@hawk.iit.edu}
}



\maketitle

\index{Hickernell, F. J.}
\index{Jim\'enez Rugama, Ll. A.}
\index{Li, D.}

\abstract{
High dimensional integrals can be approximated well by quasi-Monte Carlo methods.  
However, determining the number of function values needed to obtain the desired 
accuracy is  difficult without some upper bound on an appropriate semi-norm of the 
integrand.  This challenge has motivated our recent development of theoretically 
justified, adaptive  cubatures based on digital sequences and lattice nodeset sequences.  
Our adaptive cubatures are based on error bounds that depend on the discrete Fourier 
transforms of the integrands.  These cubatures are guaranteed for integrands belonging 
to cones of 
functions 
whose true Fourier coefficients decay steadily, a notion that is made mathematically 
precise.  
Here we describe these new cubature rules and extend them in two directions.  First, we 
generalize the error criterion to allow both absolute and relative error tolerances.  
We also demonstrate how to estimate a function of 
several 
integrals to a given tolerance.  This situation 
arises in the computation of Sobol' indices.  Second, we describe how to use control 
variates in adaptive quasi-Monte cubature while appropriately estimating the control 
variate coefficient.
}

\section{Introduction}
An important problem studied by Ian Sloan is evaluating
multivariate integrals by quasi-Monte Carlo methods.  After perhaps a change of variable, 
one may pose the problem as constructing an accurate approximation to
\begin{equation*}
\mu = \int_{[0,1)^d} f(\bsx) \, \D \bsx,
\end{equation*}
given a black-box function $f$ that provides $f(\bsx)$ for any $\bsx \in [0,1)^d$.  
Multivariate integrals arise in applications such as evaluating financial risk, computing 
multivariate probabilities, statistical physics, and uncertainty quantification.  

We have developed and implemented quasi-Monte Carlo (qMC) cubature algorithms that 
adaptively determine the sample size needed to guarantee that an error tolerance is met 
provided that the integrand belongs to a cone $\calC$ of well-behaved functions
\cite{ChoEtal15a, HicJim16a, Jim16a, JimHic16a}.  That is, given a low discrepancy 
sequence 
$\bsx_0, \bsx_1, \ldots $ and function data $f(\bsx_0), f(\bsx_1), \ldots$, we have a 
stopping rule based on the 
function data obtained so far that chooses $n$ for which
\begin{equation} \label{FTD:eq:cubprob}
\abs{\mu - \widehat{\mu}_n} \le \varepsilon, \qquad \text{where } 
\widehat{\mu}_n = \frac 1n \sum_{i=0}^{n-1} f(\bsx_i), \quad f \in \mathcal{C}.
\end{equation}
Here, $ \widehat{\mu}_n$ is the sample average of function values taken at 
well-chosen points whose empirical distribution mimics the uniform distribution.  The 
cone $\calC$ contains integrands whose Fourier coefficients decay in a reasonable 
manner, thus allowing the stopping rule to succeed.  Specifically, the size of the high 
wavenumber components of an integrand in $\calC$ cannot be large in comparison to 
the 
size of 
the low wavenumber components.  Rather 
than choosing the $\bsx_i$ to be independent and identically distributed (IID)
$\mathcal{U}[0,1)^d$ 
points, we use shifted digital sequences \cite{DicPil10a, Nie92} and sequences of 
nodesets of shifted rank-$1$ lattices \cite{HicEtal00, Mai81a, MaiSepSpa10a, 
SloJoe94}.  Sequences that are more evenly distributed than IID points are the hallmark 
of qMC algorithms.

Traditional qMC error analysis leads to error bounds of the form \cite{DicEtal14a, Hic97a}
\begin{equation*}
\abs{\mu - \widehat{\mu}_n} \le D(\{\bsx_i\}_{i=0}^{n-1}) \lVert f \rVert,
\end{equation*}
where  the integrand, $f$, is assumed to lie in some Banach space with (semi-)norm 
$\lVert \cdot \rVert$, and $\lVert f \rVert$ is often called the \emph{variation} of $f$.  
Moreover, the \emph{discrepancy} $D(\cdot)$ is a measure of quality of 
the sample, 
$\{\bsx_i\}_{i=0}^{n-1}$.  For 
integrands lying in the ball $\calB : = \{f : \lVert f \rVert \le \sigma\}$ one may construct a 
\emph{non-adaptive} algorithm 
guaranteeing $\abs{\mu - \widehat{\mu}_n} \le \varepsilon$ by choosing $n = \min \bigl 
\{n' \in \N: 
D(\{\bsx_i\}_{i=0}^{n'-1})  \le \varepsilon/\sigma \bigr \}$.

Our interest is in \emph{adaptive} qMC algorithms, where $n$ depends on the the 
function 
data observed.  Several heuristics have been proposed for choosing $n$:
\begin{description}
	\item [\textbf{Independent and identically distributed (IID) replications.} 
	\cite{Owe99a}] Compute 
	\[
	\widehat{\mu}_{n,R} = \frac 1R 
	\sum_{r=1}^R \widehat{\mu}^{(r)}_n , \qquad \widehat{\mu}_n^{(r)} = \frac 1n 
	\sum_{i=0}^{n-1} 
	f(\bsx_i^{(r)}), \quad r = 1, \ldots, R,
	\]
	where $\bigl \{\bsx_i^{(1)} \bigr \}_{i=0}^{\infty}, \ldots, \bigl \{\bsx_i^{(R)} 
	\bigr \}_{i=0}^{\infty}$ 
	are	IID randomizations of a low discrepancy 
	sequence, and $\bbE\bigl(\widehat{\mu}^{(r)}_n \bigr) = \mu$.  The standard deviation 
	of these 
	$\widehat{\mu}^{(r)}_n$, 
	perhaps 
	multiplied by an inflation vector is proposed as an upper bound for $ \lvert \mu - 
	\widehat{\mu}_{n,R}\rvert$.
	
	\item [\textbf{Internal replications.} \cite{Owe99a}] Compute 
	\[
\widehat{\mu}_{nR} = \frac 1{R} 
	\sum_{r=1}^R \widehat{\mu}^{(r)}_n =  \frac 1{nR} 
	\sum_{r=1}^{nR }f(\bsx_i), \qquad 	\widehat{\mu}_n^{(r)} = \frac 1n 
	\sum_{i=(r-1)n}^{rn-1} 
	f(\bsx_i), \quad  r = 1, \ldots, R.
	\]
	The standard deviation of these $\widehat{\mu}^{(r)}_n$, perhaps 
	multiplied by an inflation vector is proposed as an upper bound for $ \lvert \mu - 
	\widehat{\mu}_{nR}\rvert$.
	
	\item[\textbf{Quasi-standard error.} \cite{Hal05a}] Compute 
	\[
	\widehat{\mu}_{n,R} = \frac 1{R} 
	\sum_{r=1}^R \widehat{\mu}^{(r)}_n , \qquad 	\widehat{\mu}_n^{(r)} = \frac 1n 
	\sum_{i=0}^{n-1} 
	f(\bsx_{i, (r-1)d+1 : rd }), \quad  r = 1, \ldots, R,
	\]
	where $\{\bsx_i\}_{i=0}^\infty$ is now an $Rd$ dimensional sequence, and $\bsx_{i, 
	(r-1)d+1 : rd }$ denotes the $(r-1)d+1^{\text{st}}$ through $rd^{\text{th}}$ components 
	of the $i^{\text{th}}$ point in the sequence.  The standard deviation of these 
	$\widehat{\mu}^{(r)}_n$, perhaps 
	multiplied by an inflation vector is proposed as an upper bound for $ \lvert \mu - 
	\widehat{\mu}_{n,R}\rvert$.  However, see \cite{Owe06a} for cautions 
	regarding 
	this method.
\end{description}

None of the above methods have theoretical justification.  Since the proposed 
error bounds are homogeneous, it is clear that the sets of integrands for which 
these error bounds are correct are \emph{cones}.  That is, if one of the above error 
bounds above is correct for integrand $f$, it is also correct for integrand $cf$, where 
$c$ is an arbitrary constant.   Unfortunately, there is no theorem 
defining a cone $\calC$  for which any of the above error bounds  must succeed.

In this article we review our recent work developing adaptive qMC algorithms satisfying 
\eqref{FTD:eq:cubprob}.  We describe the cones $\calC$ for which our algorithms 
succeed.  We also extend our earlier algorithms in two 
directions:
\begin{itemize}
	\item Meeting more general error criteria than simply absolute error, and
	\item Using control variates to improve efficiency.
\end{itemize}
Our data-based cubature error bounds are  described in Sec.\ \ref{FTD:sec:ErrEst}.   
This section also emphasizes the similar 
algebraic structures of our two families of qMC sequences.  In Sec.\  
\ref{FTD:sec:GeneralError}, we describe how 
our error bounds can be used to satisfy error criteria that are more general 
than that in \eqref{FTD:eq:cubprob}.  Sec.\ \ref{FTD:sec:numerical} describes the 
implementation of our new adaptive qMC algorithms and provides numerical examples.  
Control variates with adaptive qMC cubature is described in Sec.\ \ref{FTD:sec:CV}.  
We conclude with a 
discussion that identifies problems for further research.

\section{Error Estimation for Digital Net and Lattice Cubature} \label{FTD:sec:ErrEst}

Here we summarize some of the key properties of cubature based on digital 
sequences and rank-$1$ lattice node sequences.  We use a 
common notation for both cases to highlight the similarities in analysis. We focus on the 
base $2$ setting for simplicity and because it is most common in practice.  Moreover, $n 
= 
2^m$ for non-negative integer $m$.  See 
\cite{HicJim16a} and \cite{JimHic16a} for more details.    

Let $\{\bszero = \bsz_0, \bsz_1, \ldots\}$ be a sequence of distinct points that is either a 
digital sequence or a rank-$1$ lattice node sequence.   Let $\oplus : [0,1)^d \times 
[0,1)^d \to [0,1)^d$ denote an addition operator under which the sequence is a group and 
the first $2^m$ points form a subgroup.  For some shift, $\bsDelta \in [0,1)^d$, 
the data sites used for cubature in 
\eqref{FTD:eq:cubprob} are given by 
$\bsx_i = \bsz_i \oplus \bsDelta$ for all $i \in \N_0$.
Typical examples of a digital sequence and a rank-$1$ lattice node sequence are given in 
Fig.\ \ref{FTD:fig:lowdiscpts}.

\begin{figure}
	\centering
	\includegraphics[width = 5.3cm]{\FTDFigDirectory/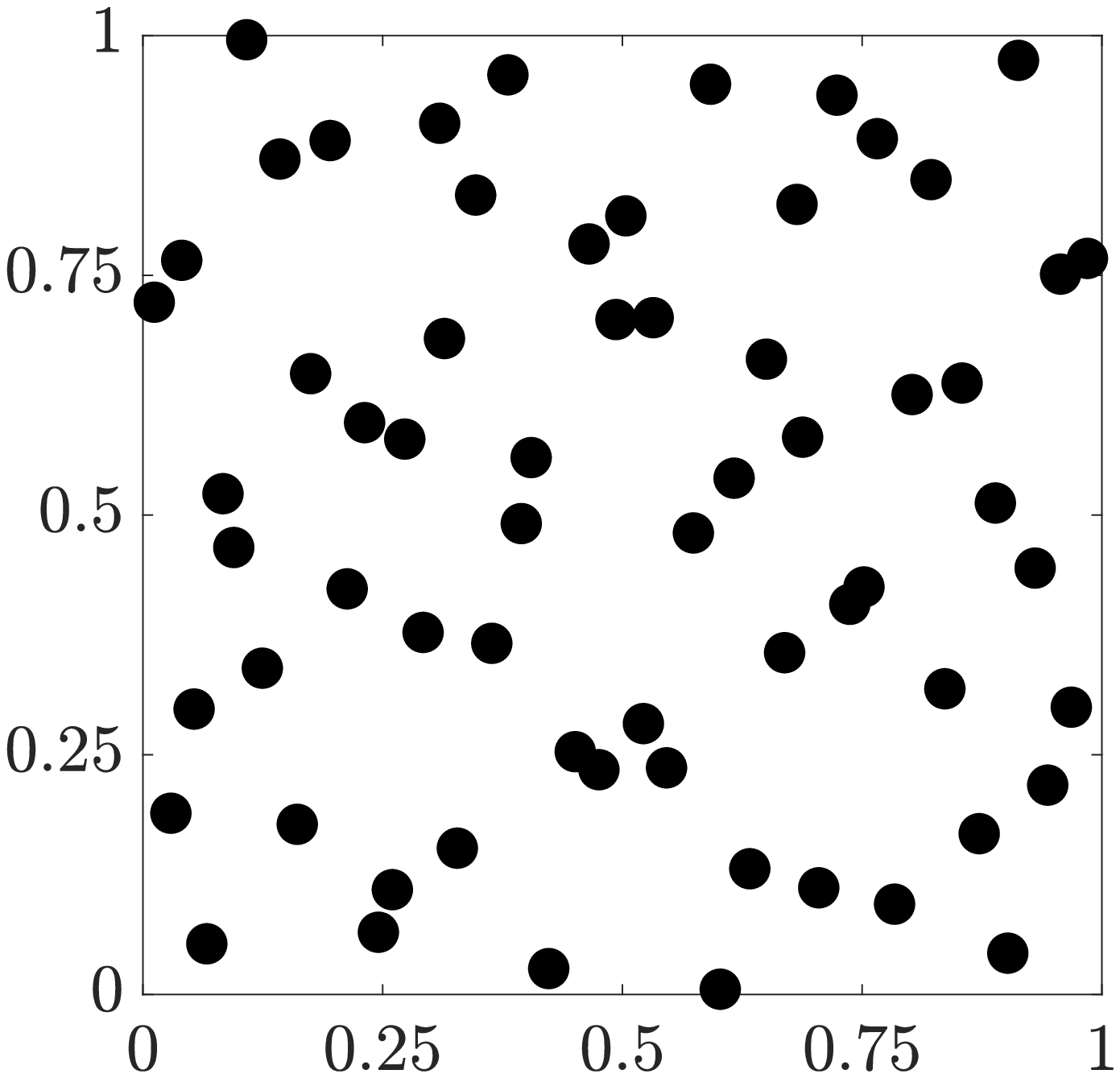} \qquad
	\includegraphics[width = 5.3cm]{\FTDFigDirectory/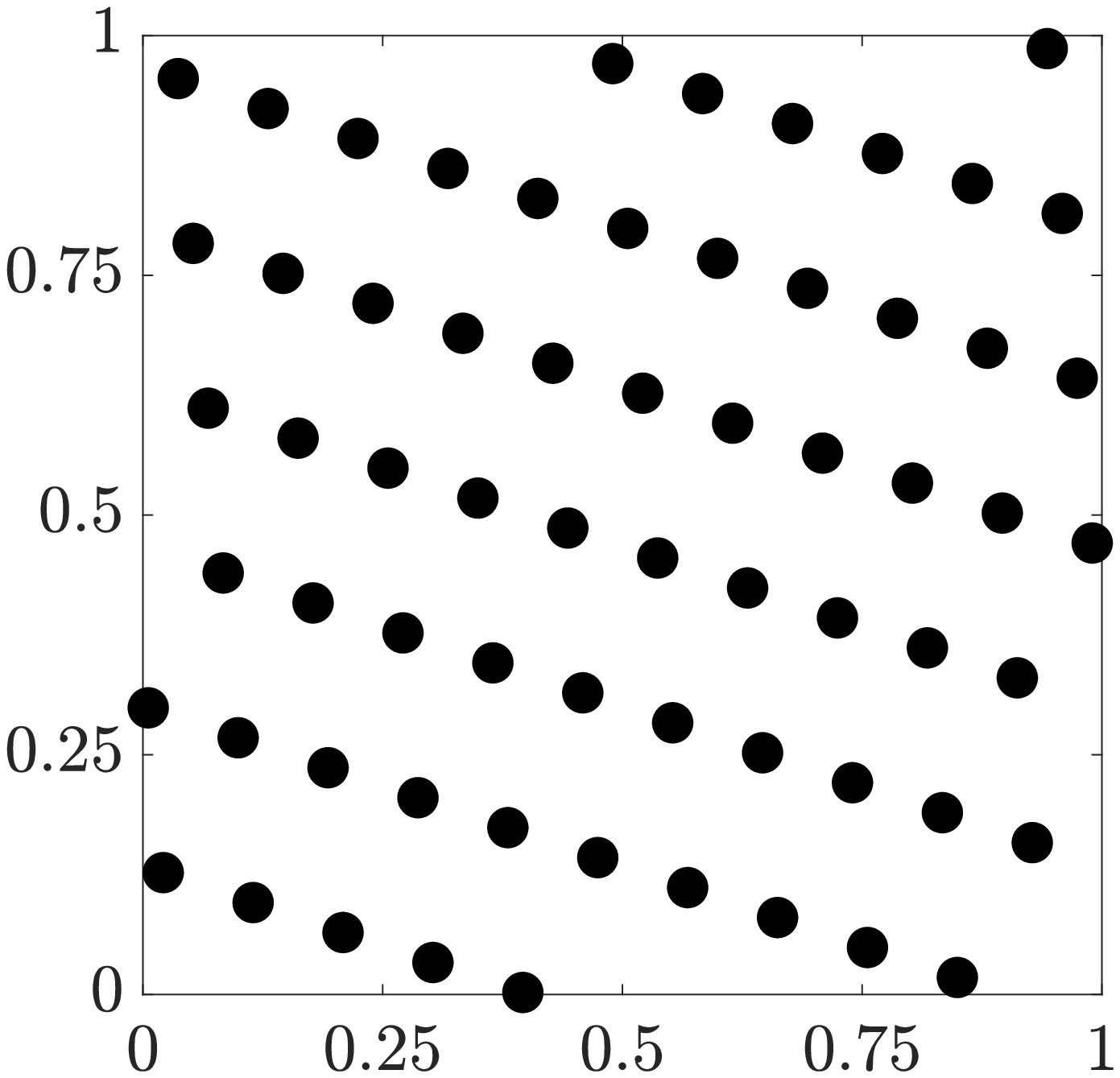}
	\caption{Two dimensional projections of a digitally shifted and Owen scrambled digital 
	sequence (left) and a 
	shifted rank-1 lattice node set (right). \label{FTD:fig:lowdiscpts}}
\end{figure}

There is a 
set of integer vector wavenumbers, $\bbK$, which is a group under its own  addition 
operator, also denoted $\oplus$.  There is also a
a bilinear functional, $\langle\cdot,\cdot \rangle: \bbK \times [0,1)^d \to \R$, which is 
used to to define a  Fourier basis for $L^2[0,1)^d$, given by $\bigl\{\E^{2 \pi \sqrt{-1} 
\langle\bsk, \cdot 	\rangle} \bigr\}_{\bsk \in \bbK}$.  The integrand is expressed as a 
Fourier series, 
\begin{multline}
f(\bsx) = \sum_{\bsk \in \bbK} \hat{f}(\bsk) \E^{2 \pi \sqrt{-1} \langle\bsk, \bsx \rangle} 
\quad  
\forall \bsx \in [0,1)^d, \ f \in L^2[0,1)^d, \nonumber \\
\text{where } \hat{f}(\bsk) : = \int_{[0,1)^d} f(\bsx)
\E^{-2 \pi \sqrt{-1} \langle\bsk, \bsx \rangle}\, \D \bsx. \label{FTD:eq:FourierDecomp}
\end{multline}
Since we require function values for cubature, we assume throughout that this Fourier 
series is absolutely convergent, i.e., $ \sum_{\bsk \in \bbK} \lvert \hat{f}(\bsk) \rvert < 
\infty$.

In the case of digital sequences, $\oplus$ denotes digit-wise addition modulo $2$ for 
points in $[0,1)^d$ and wavenumbers in $\bbK = \N_0^d$. The digits of  
$\bsz_1, \bsz_2, \bsz_4, \bsz_8, \ldots$ correspond to 
elements in the generator matrices for the usual method for constructing digital 
sequences \cite[Sec.\ 4.4]{DicPil10a}.  Also,
$\langle \bsk, \bsx 
\rangle$ is one half of an $\ell^2$ inner product of the digits of $\bsk$ and 
$\bsx$ modulo $2$. The 
$\E^{2 \pi \sqrt{-1} \langle \bsk, \cdot \rangle}$ are multivariate Walsh functions (see 
Fig.\ \ref{FTD:fig:WalshFun}).

\begin{figure}
	\centering
	\begin{tabular}{p{3cm}@{\qquad}p{3cm}@{\qquad}p{3cm}}
		\vspace{-14.75ex}\includegraphics[width=3cm]{\FTDFigDirectory/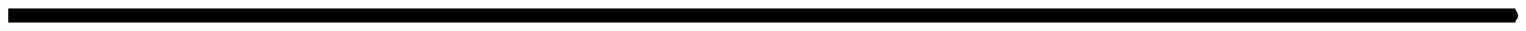} &
		\includegraphics[width=3cm]{\FTDFigDirectory/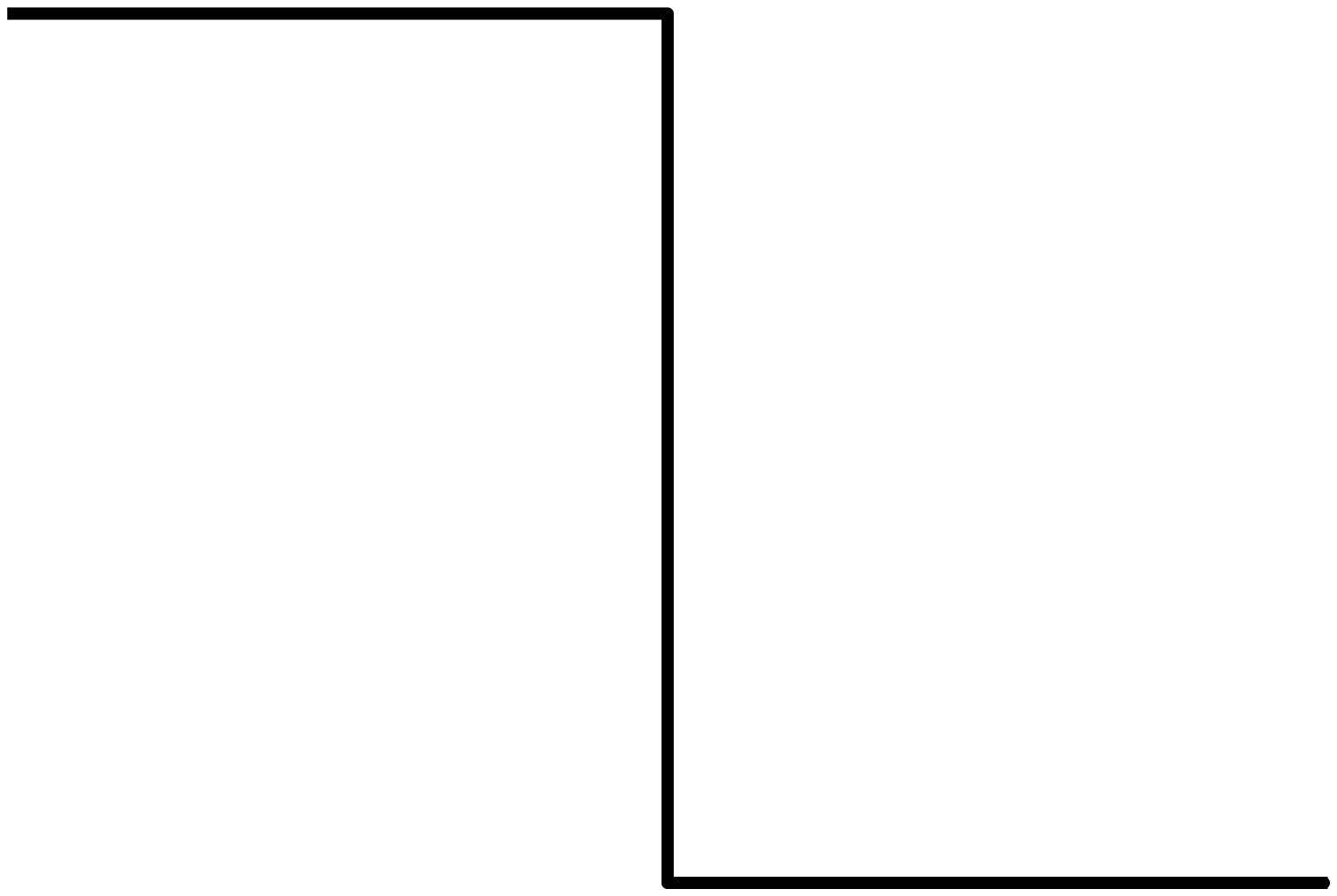} &
		\includegraphics[width=3cm]{\FTDFigDirectory/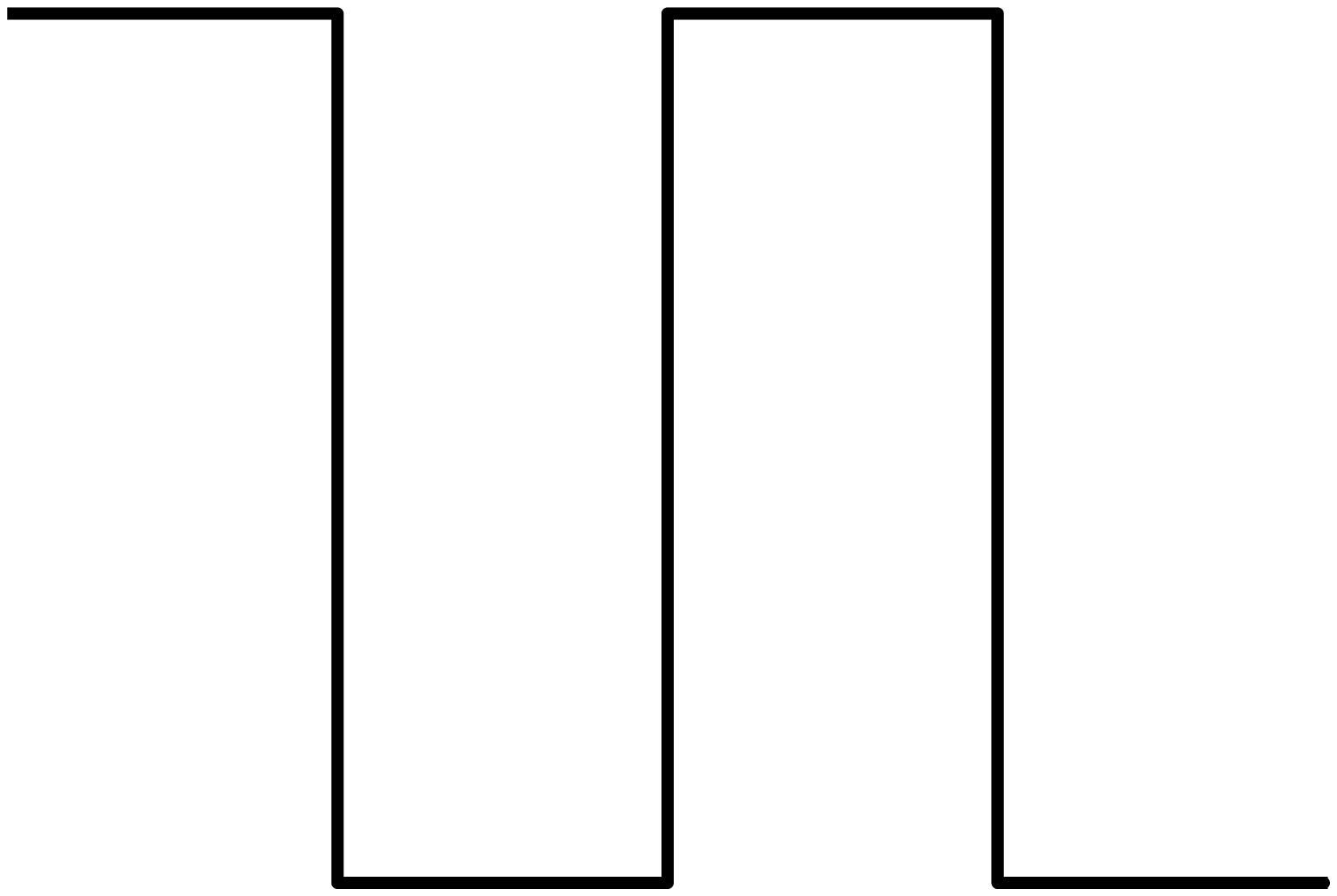} 
		\tabularnewline[2ex]
		\includegraphics[width=3cm]{\FTDFigDirectory/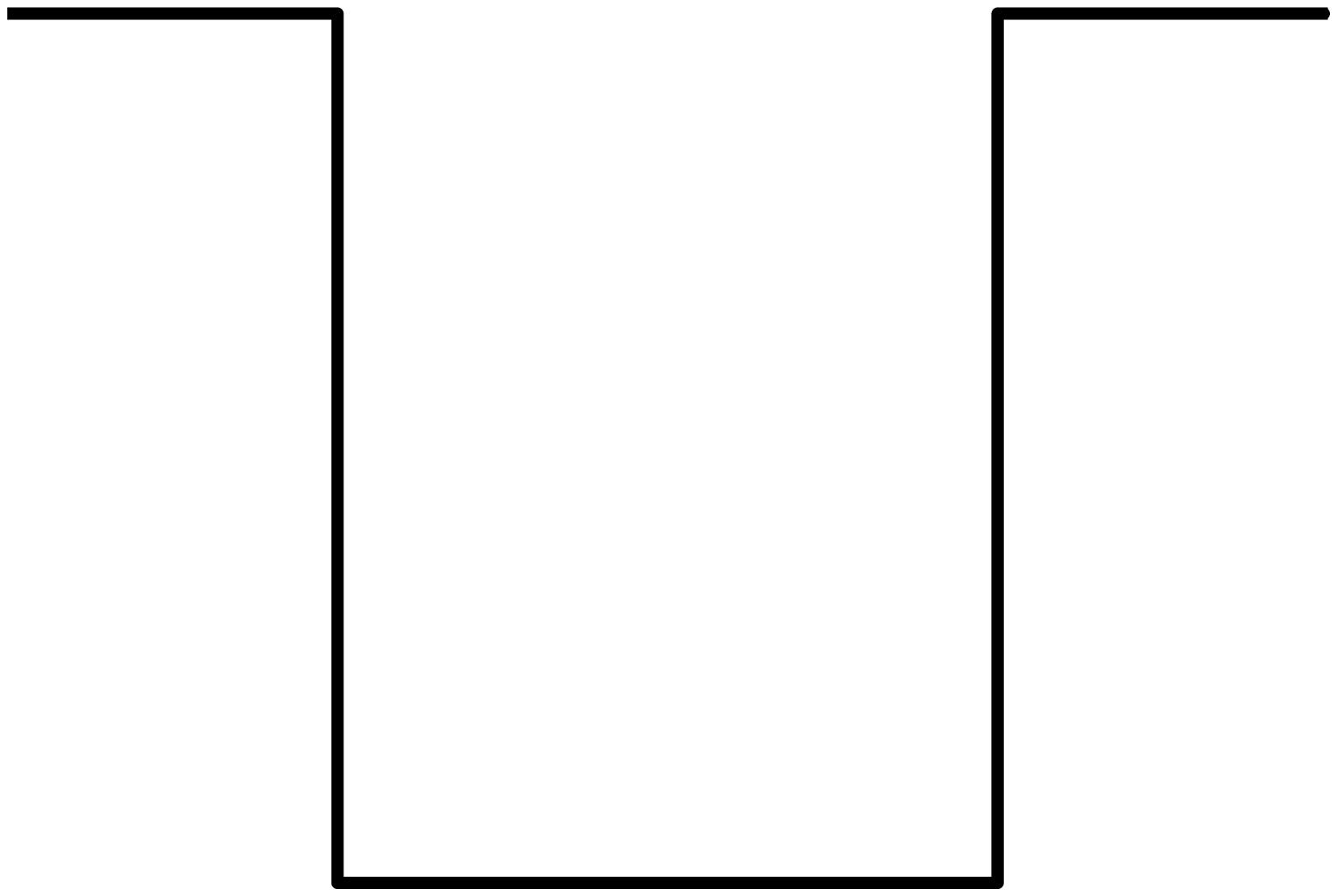} &
\includegraphics[width=3cm]{\FTDFigDirectory/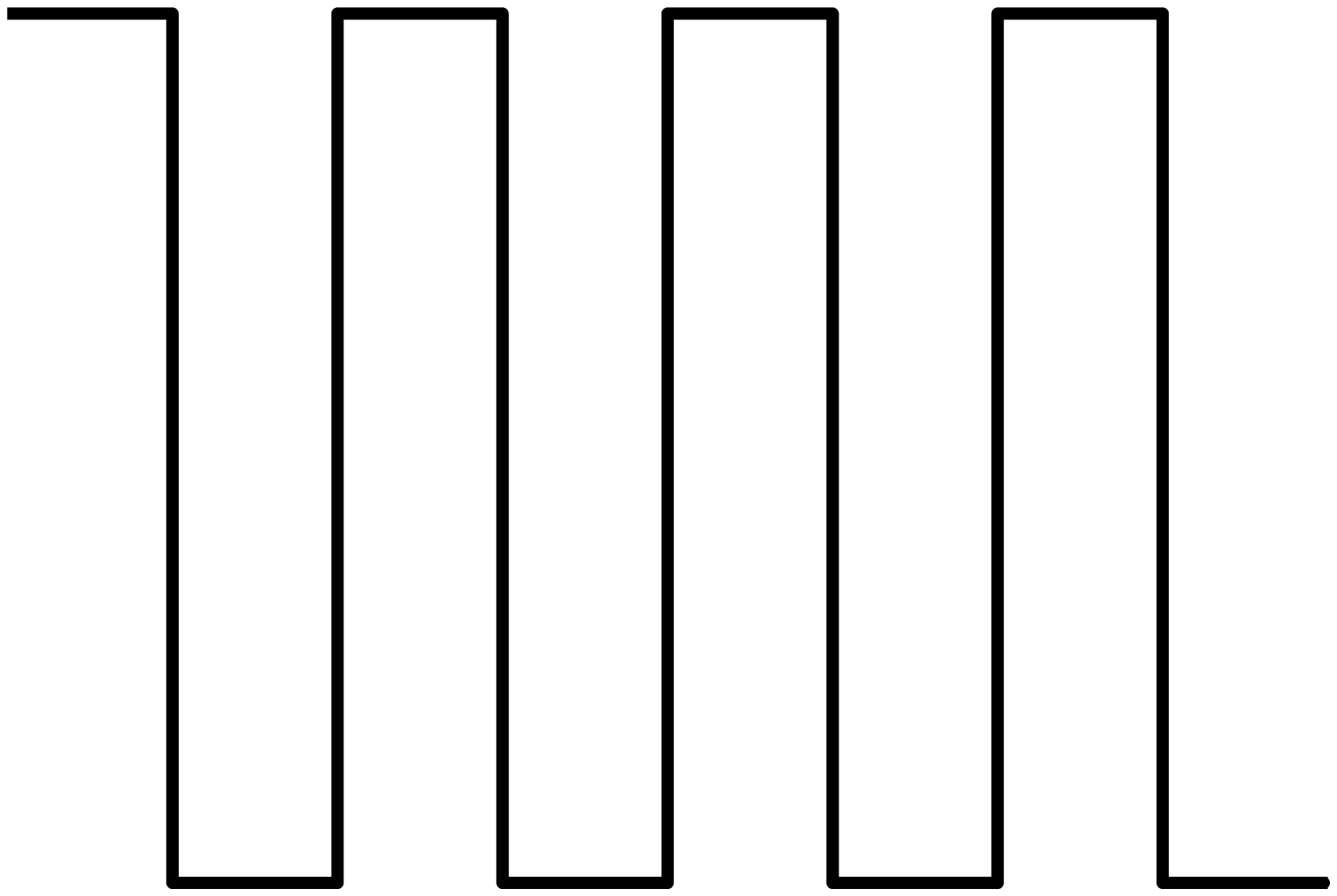} &
\includegraphics[width=3cm]{\FTDFigDirectory/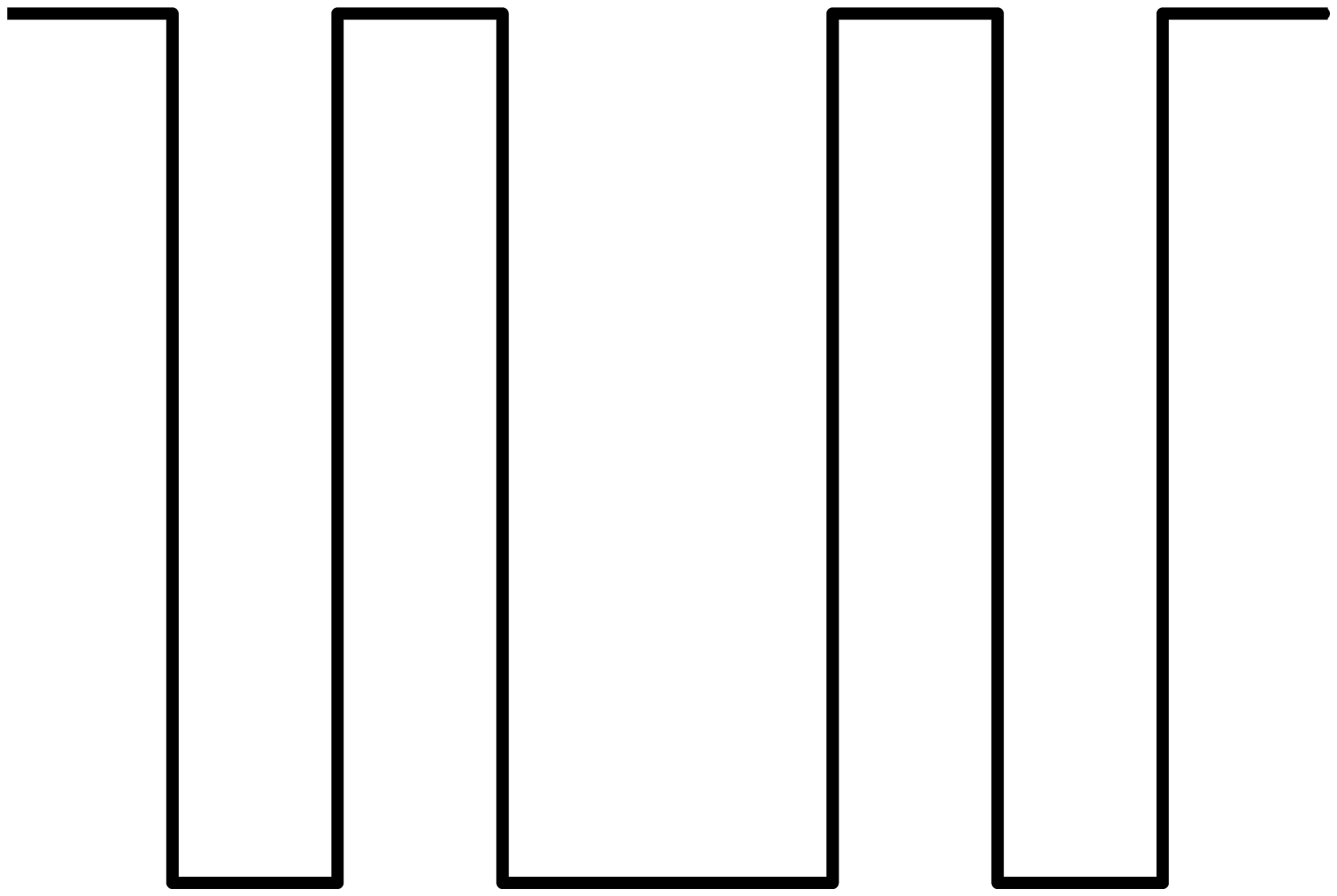}		
	\end{tabular}
	 \qquad
	\caption{One-dimensional Walsh functions corresponding to $k = 0, 1, 2,3, 4$, and 
	$5$.
	\label{FTD:fig:WalshFun}}
\end{figure}

In the case of rank-1 lattice node sequences, $\oplus$ denotes addition modulo 
$\bsone$ for  points in 
$[0,1)^d$ and ordinary addition for wavenumbers in $\bbK = \bbZ^d$.  Moreover, 
$\langle 
\bsk, \bsx \rangle = \bsk^T \bsx \bmod 1$. The 
$\E^{2 \pi \sqrt{-1} \langle \bsk, \cdot \rangle}$ are multivariate complex exponential 
functions.

The \emph{dual set} corresponding to the first $n=2^m$ unshifted points, $\{ \bsz_0, 
\ldots, 
\bsz_{2^m-1}\}$,  is 
denoted $\bbK_m$ and defined as
\begin{equation*}
\bbK_0 := \bbK, \qquad \bbK_m : =  \{\bsk \in \bbK :  \langle\bsk, \bsz_{2^\ell}\rangle = 0 
\text{ 	for all } \ell = 0, \ldots, m-1  \}, \quad m \in \N.
\end{equation*}
The dual set satisfies
\begin{equation*}
\frac{1}{2^m}\sum_{i=0}^{2^m-1} \E^{2 \pi \sqrt{-1} \langle\bsk, \bsz_i \rangle} =  
\begin{cases} 1, 
& 
\bsk \in \bbK_m, \\ 0, & \text{otherwise}. 
\end{cases}
\label{FTD:eq:dual_av_basis}
\end{equation*}

The \emph{discrete Fourier transform} of a function $f$ using $n = 2^m$ data is 
denoted 
$\tilde{f}_m$ 
and defined as 
\begin{align}
\nonumber 
\tilde{f}_m(\bsk) &: =  \frac{1}{2^m}\sum_{i=0}^{2^m-1}f(\bsx_i) \E^{-2 \pi \sqrt{-1} 
	\langle\bsk, \bsx_i \rangle} \\ 
& = \hat{f}(\bsk) + \sum_{\bsl \in \bbK_m \setminus \{\bszero\}} \hat{f}(\bsk \oplus \bsl) 
\E^{2 \pi 
	\sqrt{-1} 
	\langle\bsl, \bsDelta \rangle},   \label{FTD:eq:aliasing} 
\end{align}
after applying some of the properties alluded to above.  This last expression illustrates 
how the discrete Fourier coefficient $\tilde{f}_m(\bsk)$ 
differs from its true 
counterpart, $\hat{f}(\bsk)$, by the aliasing terms, which involve the other
wavenumbers in the coset $\bsk \oplus \bbK_m$.  As $m$ increases, 
wavenumbers leave $\bbK_m$, 
and so the 
aliasing decreases.

The sample mean of the function data is the $\bsk = \bszero$ discrete
Fourier coefficient:
\[
\widehat{\mu}_n = \frac{1}{2^m}\sum_{i=0}^{2^m-1}f(\bsx_i) = \tilde{f}_m(\bszero) = 
\sum_{\bsl \in \bbK_m} \hat{f}(\bsl)\E^{2 \pi \sqrt{-1} \langle\bsl, \bsDelta \rangle}.
\]
Hence, an error bound for the sample mean may be expressed in terms of those 
Fourier coefficients corresponding to wavenumbers in the dual set:
\begin{equation} \label{FTD:eq:error_as_Fourier}
\abs{\mu - \widehat{\mu}_n} = \abs{\hat{f}(\bszero) - \tilde{f}_m(\bszero)} = 
\abs{\sum_{\bsl \in 
		\bbK_m\setminus\{\bszero\}} \hat{f}(\bsl)\E^{2 \pi \sqrt{-1} \langle\bsl, \bsDelta 
		\rangle}} 
\leq 
\sum_{\bsl \in \bbK_m\setminus\{\bszero\}} \abs{\hat{f}(\bsl)}. 
\end{equation}

Our aim is to bound the right hand side of this cubature error bound in terms of function 
data or more specifically, in terms of the discrete Fourier transform. However,  this 
requires that the true Fourier coefficients of the integrand do not decay too erratically.  
This motivates our definition of $\calC$, the cone of integrands for which our adaptive 
algorithms succeed.  

To facilitate the definition of $\calC$
we construct an ordering of the wavenumbers, $\tilde{\bsk}: 
\N_0 \to \mathbb{K}$ satisfying $\tilde{\bsk}(0) = \bszero$ and 
$\bigl \{\tilde{\bsk}(\kappa + \lambda 2^m) \bigr\}_{\lambda =0}^{\infty}= 
\tilde{\bsk}(\kappa) \oplus 
\bbK_m$ for $\kappa  = 0, \ldots, 2^{m}-1$ and $m \in \N_0$, as described in 
\cite{HicJim16a,JimHic16a}.  This condition implies the crucial fact that 
$\abs{\tilde{f}_m(\tilde{\bsk}(\kappa + \lambda 
2^m))}$ is the same for all $\lambda \in \N_0$.  Although there is some arbitrariness in 
this ordering, it is understood that $\tilde{\bsk}(\kappa)$ generally increases in 
magnitude as 
$\kappa$ tends to infinity.  We adopt the shorthand notation 
$\hat{f}_{\kappa}:=\hat{f}(\tilde{\bsk}(\kappa))$ and 
$\tilde{f}_{m,\kappa}:=\tilde{f}_m(\tilde{\bsk}(\kappa))$.  
Then, the
error bound in  \eqref{FTD:eq:error_as_Fourier} may be written as 
\begin{equation} \label{FTD:eq:error_as_Fourier_kappa}
\abs{\mu - \widehat{\mu}_n} 
		\leq 
\sum_{\lambda = 1}^{\infty} \abs{\hat{f}_{\lambda 2^m}}. 
\end{equation}

The cone of functions whose Fourier series are absolutely convergent and whose true 
Fourier coefficients, $\hat{f}_\kappa$, decay steadily as $\kappa$ tends to 
infinity is
\begin{subequations} \label{FTD:eq:cone_def}
\begin{multline}
\mathcal{C}=\{f \in AC([0,1)^d) : \widehat{S}_{\ell,m}(f) \le \widehat{\omega}(m-\ell) 
\widecheck{S}_m(f),\ \ 
\ell \le 
m, \\
\widecheck{S}_m(f) \le \mathring{\omega}(m-\ell) S_{\ell}(f),\ \  \ell_* \le \ell \le m\},
\end{multline}
where
\begin{gather}
\label{FTD:eq:SShatdef}
S_m(f) :=  \sum_{\kappa=\left \lfloor 2^{m-1} \right \rfloor}^{2^{m}-1} 
\abs{\hat{f}_{\kappa}}, 
\qquad 
\widehat{S}_{\ell,m}(f)  := \sum_{\kappa=\left \lfloor 2^{\ell-1} \right \rfloor}^{2^{\ell}-1} 
\sum_{\lambda=1}^{\infty} \abs{ \hat{f}_{\kappa+\lambda 2^{m}}}, \\
\label{FTD:eq:Scheckdef}
\widecheck{S}_m(f):=\widehat{S}_{0,m}(f) + \cdots + \widehat{S}_{m,m}(f)=
\sum_{\kappa=2^{m}}^{\infty} \abs{\hat{f}_{\kappa}},
\end{gather}
\end{subequations}
and where $\ell, m \in \N_0$ and $\ell \le m$.  The positive integer $\ell_*$ and the 
bounded 
functions $\widehat{\omega}, \mathring{\omega}: 
\N_0 \to [0,\infty)$ are parameters that determine how inclusive $\calC$ is and  how 
robust our algorithm is.  Moreover, $\mathring{\omega} (m) \to 0$ as $m \to \infty$.  The 
default values are provided in Sec.\ 
\ref{FTD:sec:numerical}.

We now explain the definition of the cone $\calC$ and the data driven cubature error 
bound that we are able to derive.  For illustration we use the functions depicted in 
Fig.\ \ref{FTD:fig:inoutconef}.  The one on the left lies inside $\calC$ because its 
Fourier coefficients decay steadily (but not necessarily monotonically), while the one on 
the right lies outside $\calC$ because its Fourier coefficients decay erratically.  The 
function lying outside $\calC$ resembles the one lying inside $\calC$ but with high 
wavenumber noise.

The sum of the absolute value of the Fourier coefficients appearing on the 
right side of error bound \eqref{FTD:eq:error_as_Fourier_kappa} is $\widehat{S}_{0,m}(f)$ 
according to the definition in \eqref{FTD:eq:SShatdef}.  In Fig.\  
\ref{FTD:fig:inoutconef}, $m = 11$, and $\widehat{S}_{0,11}(f)$ corresponds to the sum 
of $\lvert \hat{f}_\kappa \rvert$ for $\kappa = 2~048, 4~096, 6~144, \ldots $.  Since only $n 
= 2^m$ 
function 
values are available, it is impossible to estimate the Fourier coefficients appearing in 
$\widehat{S}_{0,m}(f)$ directly by discrete Fourier coefficients.  

By the 
definition in 
\eqref{FTD:eq:Scheckdef}, it follows that $\widehat{S}_{0,m}(f) \le 
\widecheck{S}_m(f)$.   In Fig.\  
\ref{FTD:fig:inoutconef}, $\widecheck{S}_{11}(f)$ corresponds to the sum 
of all $\lvert \hat{f}_\kappa \rvert$ with $\kappa \ge 2048$.
The definition of $\calC$ assumes that $\widehat{S}_{0,m}(f) \le \widehat{\omega}(m) 
\widecheck{S}_m(f)$, where 
$\widehat{\omega}(m)$ could be chosen as $1$ or could decay with $m$.  This is up to 
the user.  

\begin{figure}
	\centering
	\begin{tabular}{p{5.3cm}@{\qquad}p{5.3cm}}
		\includegraphics[width=5.3cm]{\FTDFigDirectory/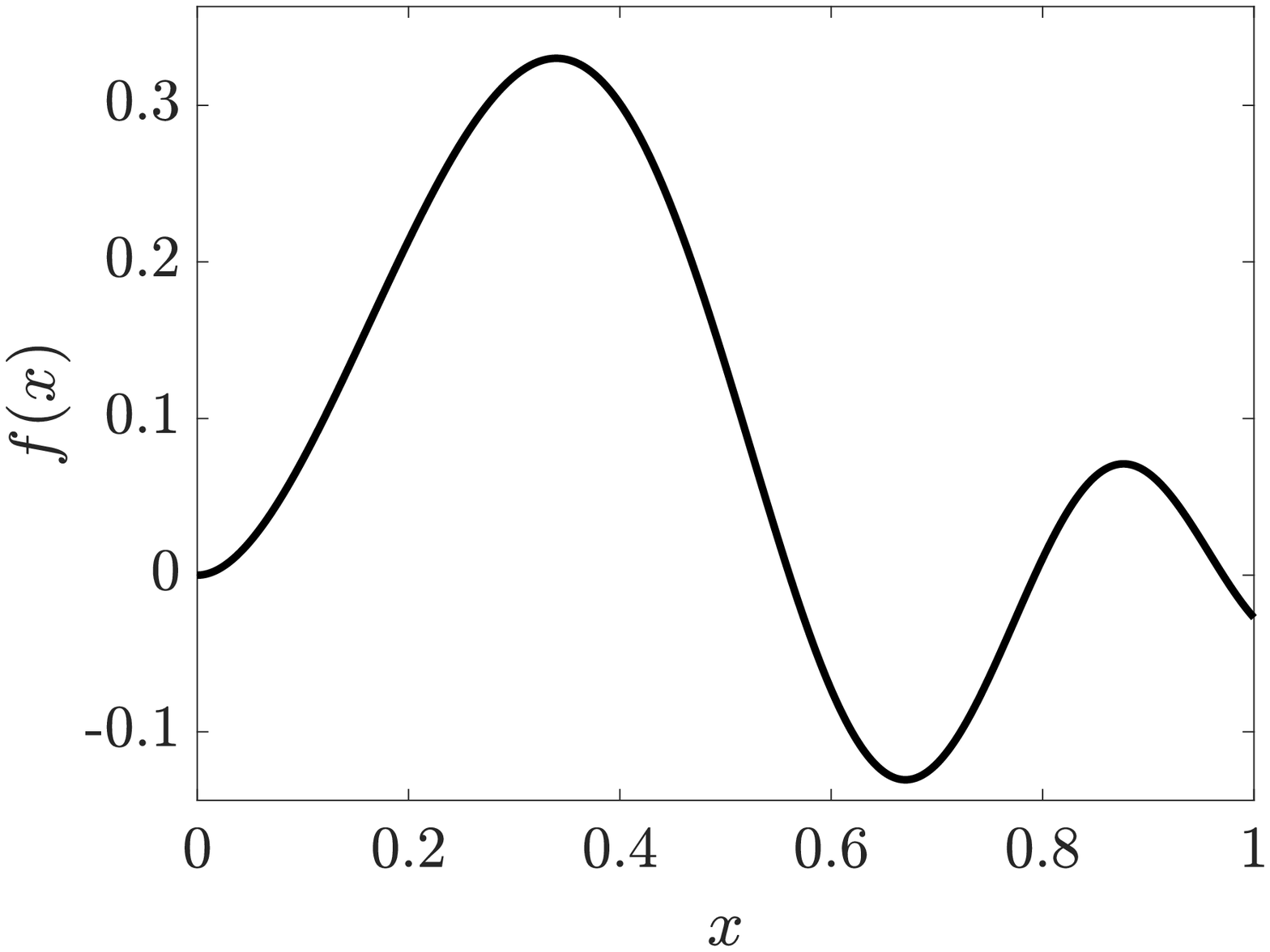}
		&			
		\includegraphics[width=5.3cm]{\FTDFigDirectory/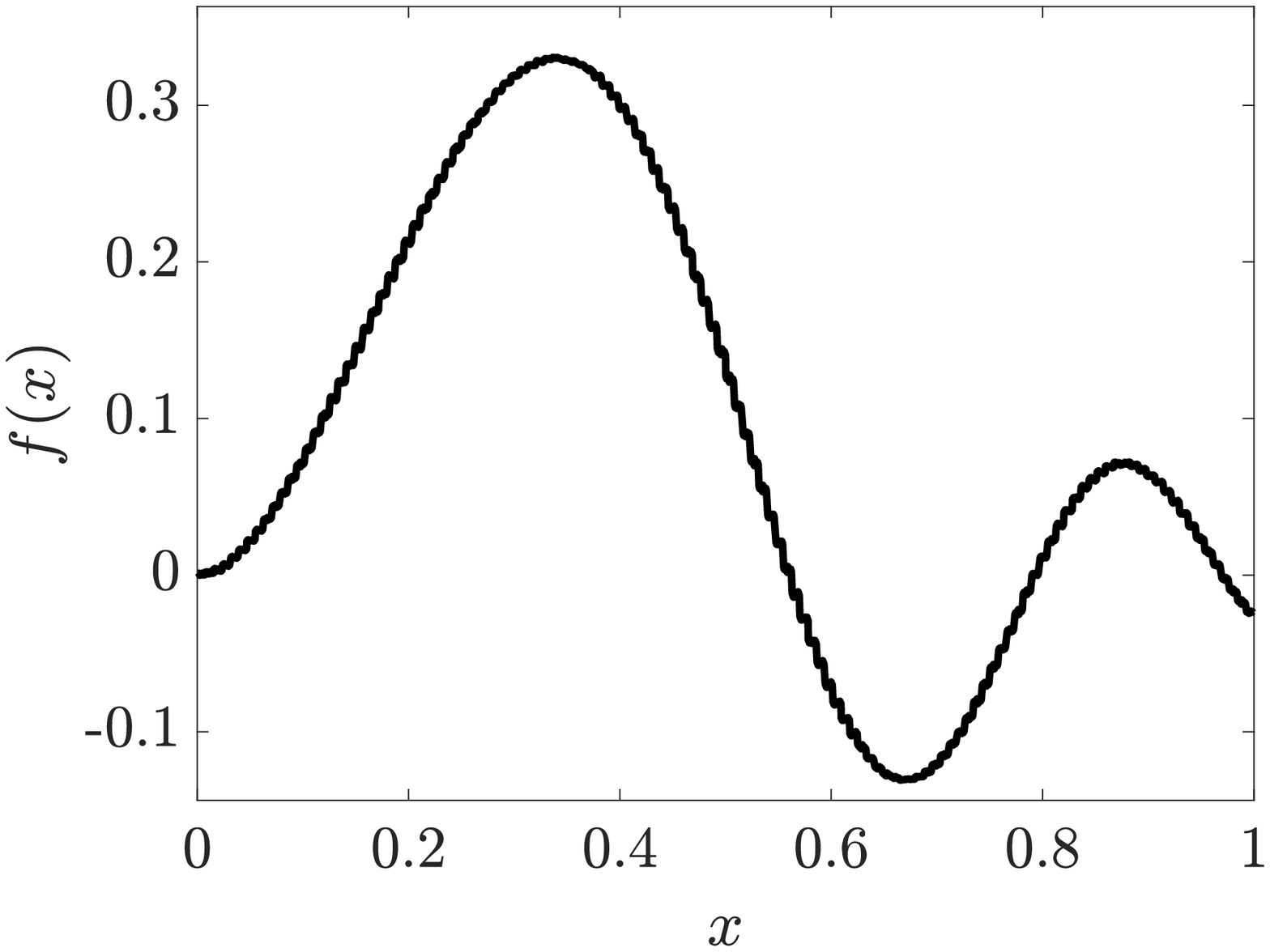}
		\tabularnewline
		\includegraphics[width=5.3cm]{\FTDFigDirectory/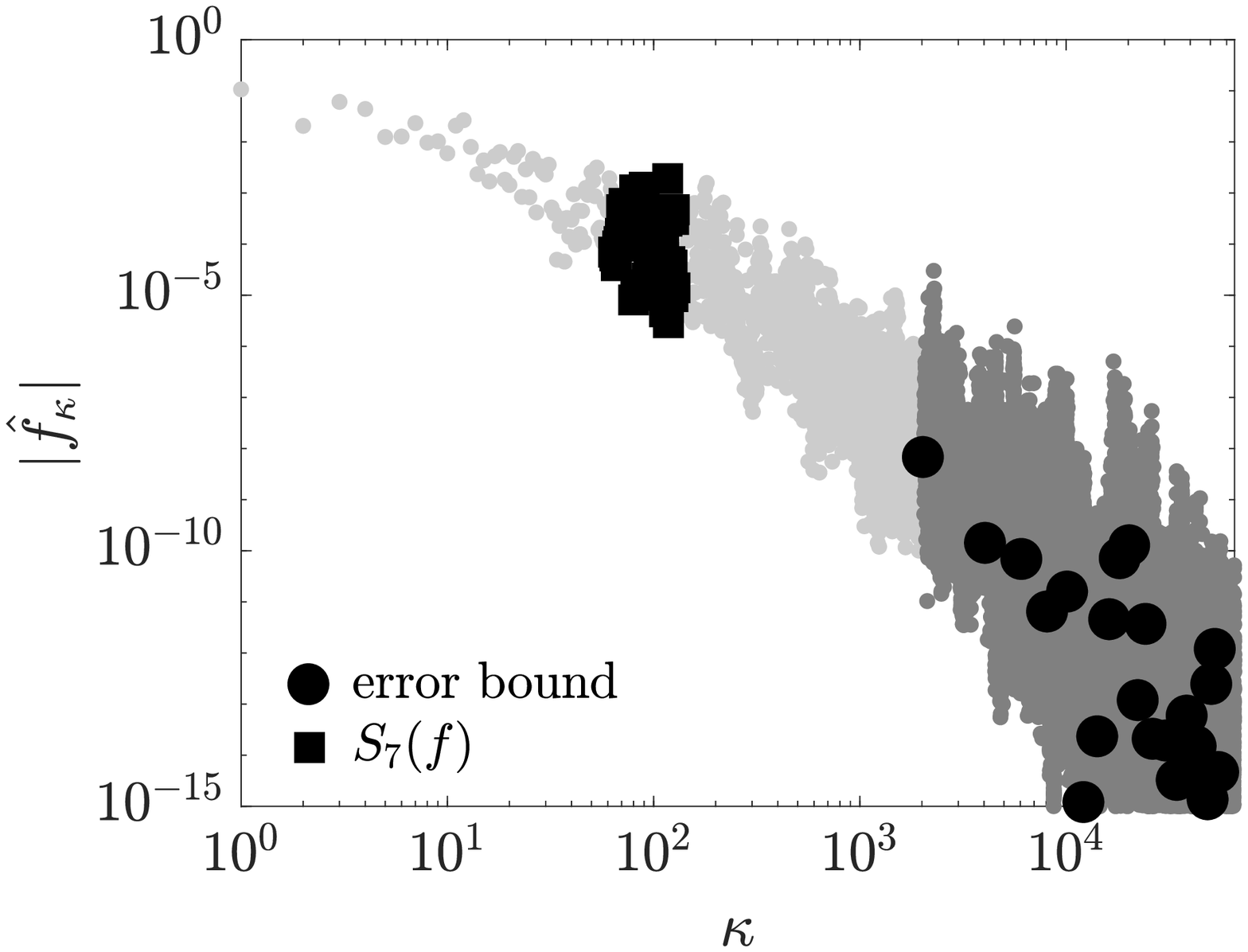}
		&			
		\includegraphics[width=5.3cm]{\FTDFigDirectory/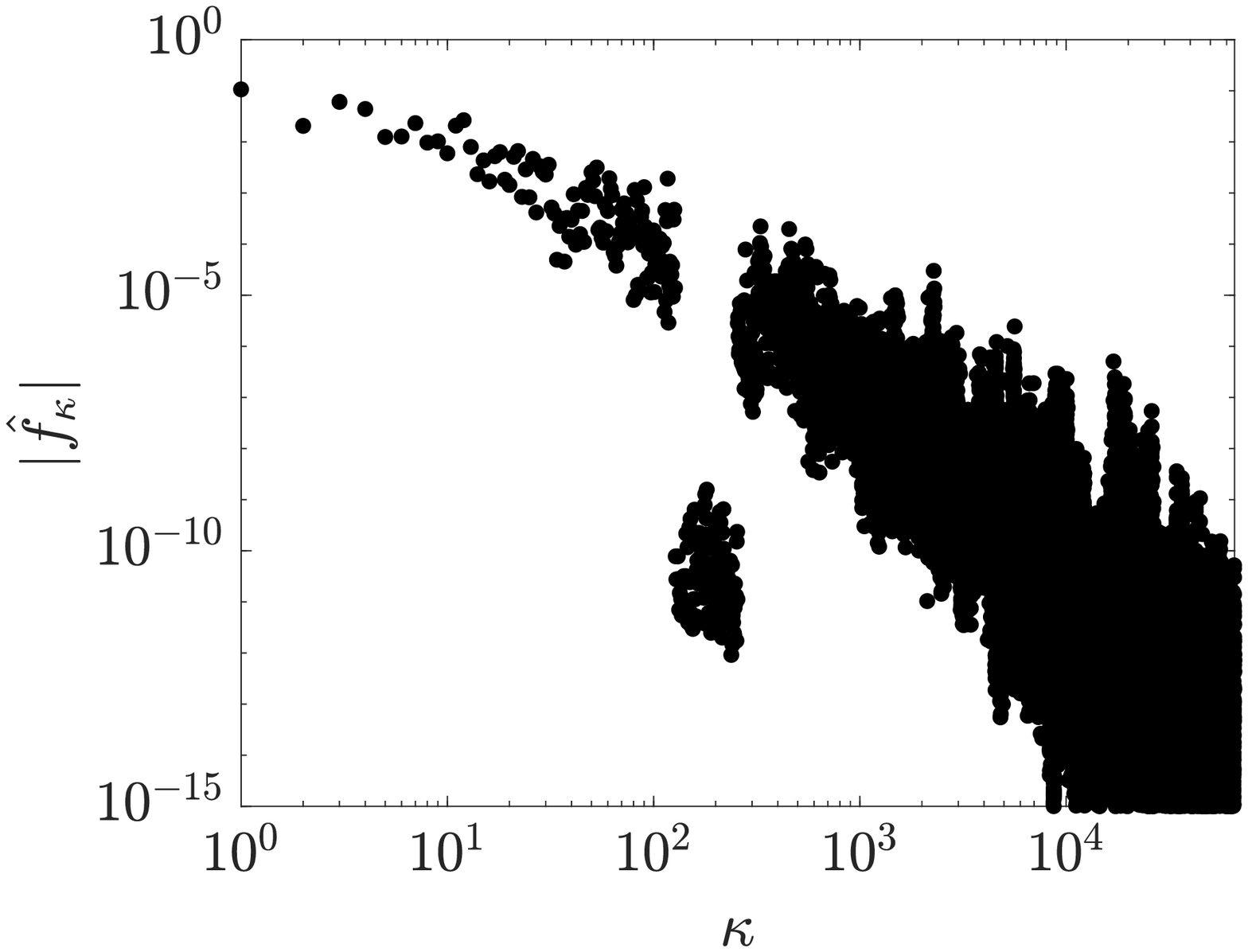}
	\end{tabular}
	\caption{A typical function lying inside $\calC$ and its Fourier Walsh coefficients (left) 
	in 
		contrast to a function lying outside $\calC$ and its Fourier Walsh coefficients (right).
		\label{FTD:fig:inoutconef}}
\end{figure}

Still, $ \widecheck{S}_m(f)$ involves Fourier coefficients that are of too high a 
wavenumber to 
be approximated by discrete Fourier coefficients.  The definition of $\calC$ also assumes 
that $\widecheck{S}_{m}(f) \le \mathring{\omega}(r) S_{m-r}(f)$ for any non-negative $r 
\le m -  
\ell_*$.  
This means that the infinite sum of the high wavenumber coefficients, 
$\widecheck{S}_{m}(f)$ 
cannot 
exceed some factor, $\mathring{\omega}(r)$, of the finite sum of modest wavenumber 
coefficients 
$S_{m-r}(f)$.  In Fig.\ \ref{FTD:fig:inoutconef}, $r = 4$, and the graph on the left shows 
$\widecheck{S}_{11}(f)$ to be bounded 
above by $\mathring{\omega}(4) S_{7}(f)$ for a modest value of 
$\mathring{\omega}(4)$.  Recall from 
the definition in \eqref{FTD:eq:SShatdef} 
that $ S_{7}(f)$ is the sum of the absolute value of the Fourier coefficients corresponding 
to $64, \ldots, 127$.  However, the function depicted in the right of Fig.\ 
\ref{FTD:fig:inoutconef} violates the assumption that $\widecheck{S}_{11}(f) \le 
\mathring{\omega}(4) 
S_{7}(f)$ because $S_{7}(f)$ in that case is very small.  Thus, the function on the right in 
Fig.\ \ref{FTD:fig:inoutconef} lies outside $\calC$.

Based on the above argument, it follows in general that for $f \in \calC$,
\begin{multline}
\label{FTD:eq:step1}
\sum_{\lambda = 1}^{\infty} \abs{\hat{f}_{\lambda 2^m}} = \widehat{S}_{0,m}(f) \le 
\widehat{\omega}(m) 
\widecheck{S}_m(f) \le \widehat{\omega}(m)  \mathring{\omega}(r) S_{m-r}(f),  \\ m \ge r 
+ \ell_* 
\ge \ell_*.
\end{multline}
This implies an error bound in terms of the true Fourier coefficients with modest 
wavenumber.  In 
particular \eqref{FTD:eq:step1} holds for the function depicted on the left side of Fig.\ 
\ref{FTD:fig:inoutconef}, but not the one on the right side.

Before going on, we note that we have not specified the parameters  $\ell_*, r, 
\widehat{\omega}$, and $\mathring{\omega}$ for the sake of simplicity.  Their choices 
reflect 
the 
robustness desired by the user, but are meant to be kept constant rather than changed 
for every problem.  The parameter $\ell_*$ 
is the minimum wavenumber for which we expect steady decay to set in.  The parameter 
$r$ controls how small the values of the wavenumber that are used to bound 
the cubature error should be.  The functions $\widehat{\omega}$ and 
$\mathring{\omega}$ are 
the 
inflation factors 
for bounding one sum of Fourier coefficients in terms of another.  See Sec.\ 
\ref{FTD:sec:numerical} for the default choices in 
our algorithm implementations.

While \eqref{FTD:eq:step1} is a step forward, it involves the unknown true Fourier 
coefficients and not the known discrete Fourier coefficients.   We next bound  
$S_{m-r}(f)$ in terms of a sum of discrete Fourier 
coefficients:
\begin{equation*}
\widetilde{S}_{\ell,m}(f) := \sum_{\kappa=\left \lfloor 2^{\ell-1}\right \rfloor}^{2^{\ell}-1} 
\abs{\tilde{f}_{m,\kappa}}.
\end{equation*}
By  \eqref{FTD:eq:aliasing} and the triangle inequality it follows that 
\begin{align}
\nonumber
\widetilde{S}_{m-r,m}(f) & =  \sum_{\kappa=\left \lfloor 2^{m-r-1} \right \rfloor}^{2^{m-r}-1} 
\abs{\tilde{f}_{m,\kappa}} \\
\nonumber
& \ge  \sum_{\kappa=\left \lfloor 2^{m-r-1} \right \rfloor}^{2^{m-r}-1} 
\biggl[\abs{\hat{f}_{\kappa}} - 
\sum_{\lambda = 1}^{\infty}\abs{\hat{f}_{\kappa + \lambda 2^m}}  \biggr] \\
\nonumber
& = S_{m-r}(f) -  \widehat{S}_{m-r,m}(f) \\
& \ge S_{m-r}(f)[1 - \widehat{\omega}(r)  \mathring{\omega}(r)]. 
\label{FTD:eq:step2}
\end{align}
This provides an upper bound on $S_{m-r}(f)$ in terms of the data-based 
$\widetilde{S}_{m-r,m}(f)$, provided that $r$ is large enough to satisfy 
$\widehat{\omega}(r) 
\mathring{\omega}(r)<1$.  Such a choice of $r$ ensures that the aliasing errors are 
modest.  

Combining \eqref{FTD:eq:step1} and \eqref{FTD:eq:step2} with 
\eqref{FTD:eq:error_as_Fourier_kappa}, it is shown in  
\cite{HicJim16a,JimHic16a} 
that  for any $f\in\mathcal{C}$, 
\begin{equation} \label{FTD:eq:data_err_bd}
\abs{\mu - \widehat{\mu}_n} \le  \textup{err}_n : = \frakC(m) \widetilde{S}_{m-r,m}(f), 
\quad   
\frakC(m)  : = 
\frac{\widehat{\omega}(m) 
\mathring{\omega}(r)}{1 - 
\widehat{\omega}(r) 
\mathring{\omega}(r)}, \qquad m \ge \ell_*+r,
\end{equation}
provided that $\widehat{\omega}(r) \mathring{\omega}(r)<1$.  Since 
$\widetilde{S}_{m-r,m}(f)$ 
depends 
only on the 
discrete Fourier coefficients, \eqref{FTD:eq:data_err_bd} is a data-based cubature error 
bound.  One may now increment $m$ (keeping $r$ fixed) until $\textup{err}_n$ is small 
enough, 
where again $n = 2^m$.  

If $\$(f)$ denotes the cost of one function value, then evaluating $f(\bsx_0), 
\ldots, f(\bsx_{2^m-1})$ requires $\$(f)n$ operations.  A fast 
transform then computes $\tilde{f}_{m,0}, \ldots, \tilde{f}_{m,2^m-1}$
in an additional $\calO(n \log(n)) = \calO(m 2^m)$ operations.  So computing 
$\textup{err}_{2^m}$ for each $m$
costs 
$\calO\bigl([\$(f) + m] 2^m\bigr)$ operations.  For 
integrands 
that are cheap to evaluate the $\$(f)$ term is negligible, but for integrands that are 
expensive to integrate $\$(f)$ may be comparable to $m$ given that $m$ might be 
ten 
to twenty.

Using an analogous reasoning as in \eqref{FTD:eq:step2},
\begin{align}
\nonumber
S_\ell(f) & = \sum_{\kappa=\left\lfloor 2^{\ell-1} \right\rfloor}^{2^{\ell}-1} 
\abs{\hat{f}_{\kappa}} \\
\nonumber
& \ge \sum_{\kappa=\left \lfloor 2^{\ell-1} \right \rfloor}^{2^{\ell}-1} 
\biggl[\abs{\tilde{f}_{m,\kappa}} - 
\sum_{\lambda = 1}^{\infty}\abs{\hat{f}_{\kappa + \lambda 2^m}}  \biggr] \\
\nonumber
& = \widetilde{S}_{\ell,m}(f) -  \widehat{S}_{\ell,m}(f) \\
& \ge \widetilde{S}_{\ell,m}(f)/[1 + \widehat{\omega}(m-\ell)  \mathring{\omega}(m-\ell)]. 
\label{FTD:eq:step3NecCond}
\end{align}
Therefore, from \eqref{FTD:eq:step2} and \eqref{FTD:eq:step3NecCond}, for any $\ell, 
m,m'\in\N$ such that $\ell_*\leq\ell\leq \min(m,m')$, it must be the case that
\begin{equation}\label{FTD:eq:NecCond}
\frac{\widetilde{S}_{\ell,m}(f)}{1 + \widehat{\omega}(m-\ell) \mathring{\omega}(m-\ell)} \le 
S_\ell(f) 
\le 
\frac{\widetilde{S}_{\ell,m'}(f)}{1 - \widehat{\omega}(m'-\ell) \mathring{\omega}(m'-\ell)}.
\end{equation}
Equation \eqref{FTD:eq:NecCond} is a data-based \emph{necessary} condition for an 
integrand, 
$f$,  
to lie in $\mathcal{C}$.  If it is found that the right hand side of  
\eqref{FTD:eq:NecCond} is smaller than the left hand side of  \eqref{FTD:eq:NecCond}, 
then $f$ must lie outside $\mathcal{C}$.  In this case the parameters defining the cone 
should be adjusted to expand the cone appropriately, e.g., by increasing 
$\widehat{\omega}$ or $\mathring{\omega}$ by a constant.

By substituting inequality \eqref{FTD:eq:step3NecCond} in the error bound 
\eqref{FTD:eq:data_err_bd}, we get
\[
\textup{err}_n \leq \frac{\widehat{\omega}(m) \mathring{\omega}(r)}{1 - 
\widehat{\omega}(r) 
\mathring{\omega}(r)}[1 + \widehat{\omega}(r) \mathring{\omega}(r)]S_{m-r}(f).
\]
We define $m^*$,
\begin{equation} \label{FTD:eq:abserrcost}
m^*:=\min\left\{m\geq \ell_*+r:\; \frac{\widehat{\omega}(m) \mathring{\omega}(r)}{1 - 
\widehat{\omega}(r) \mathring{\omega}(r)}[1 + \widehat{\omega}(r) 
\mathring{\omega}(r)]S_{m-r}(f)\leq\varepsilon \right\},
\end{equation}
Here $m^*$ depends on the fixed parameters of the algorithm, $\ell_*, r, 
\widehat{\omega},$ and 
$\mathring{\omega}$.   Note that 
$\textup{err}_{2^{m^*}} \le \varepsilon$.

Recall from above that at each step $m$ in our algorithm the  computational cost is 
$\calO\bigl([\$(f)+m]2^{m} \bigr)$.  Thus, the computational cost for our adaptive 
algorithm 
to satisfy the absolute error tolerance, as given in \eqref{FTD:eq:cubprob}, is 
$\calO(\Phi(m^*)2^{m^*})$, where $\Phi(m^*) = 
[\$(f)+ 0]2^{-m^*} + \cdots + [\$(f)+ m^*]2^{0}$.  Since 
\begin{equation*}
\Phi(m^*+1) - \Phi(m^*) = \$(f)2^{-m^*-1} + 2^{-m^*} + \cdots + 2^{0} \le 
\$(f)2^{-m^*-1} + 2,
\end{equation*}
it follows that 
\begin{align*}
\Phi(m^*) & = [\Phi(m^*) - \Phi(m^*-1)] + \cdots + [\Phi(1) - \Phi(0)] \\
& \le [\$(f)2^{-m^*} + 2] + \cdots + [\$(f)2^{-1} + 2] \\
& \le 2[\$(f) + m^*]
\end{align*}
Thus, the cost of making our 
data based error bound no greater than 
$\varepsilon$ is bounded above by $\calO\bigl([\$(f)+m^*]2^{m^*} \bigr)$. 

The algorithm does not assume a rate of decay of the Fourier coefficients but 
automatically senses the rate of decay via the discrete Fourier coefficients.   From 
\eqref{FTD:eq:abserrcost} it is evident that the dependence of the computational cost 
with $\varepsilon$ depends primarily on the unknown rate of decay of $S_{m-r}(f)$ with 
$m$, and secondarily on the specified rate of decay of $\widehat{\omega}(m)$, since all 
other 
parameters are fixed.  For example, assuming $\widehat{\omega}(m) = \calO(1)$, if 
$\hat{f}_\kappa = 
\calO(\kappa^{-p})$, then $S_{m-r}(f) = 
\calO(2^{-(p-1)m}) $, and the total computational cost is 
$\calO(\varepsilon^{-1/(p-1) - 
\delta})$ for all $\delta > 0$.  If $\widehat{\omega}(m)$ decays with $m$, then the 
computational 
cost is less.

\section{General Error Criterion} \label{FTD:sec:GeneralError}

The algorithms summarized above are described  in 
\cite{HicJim16a,JimHic16a} and implemented in the 
Guaranteed Automatic Integration Library (GAIL) \cite{ChoEtal15a} as 
\texttt{cubSobol\_g} 
and \texttt{cubLattice\_g}, respectively.  They
satisfy  the absolute error criterion \eqref{FTD:eq:cubprob} by 
increasing $n$ until $\textup{err}_n$ defined in \eqref{FTD:eq:data_err_bd} is no greater 
than the 
absolute error tolerance, $\varepsilon$.  

There are situations requiring a more general 
error criterion than \eqref{FTD:eq:cubprob}.  In this section we generalize 
the cubature problem to involve a $p$-vector of integrals, $\bsmu$, which are 
approximated by a $p$-vector of sample means, $\widehat{\bsmu}_n$, using $n$ 
samples, and for which we have a $p$-vector of error bounds, $\textbf{err}_n$, given by 
\eqref{FTD:eq:data_err_bd}.  This 
means that $\bsmu \in [\widehat{\bsmu}_n - \textbf{\textup{err}}_n, 
\widehat{\bsmu}_n + \textbf{\textup{err}}_n]$ for integrands in $\calC$.  Given some 
\begin{itemize}
	\item function $v: \Omega \subseteq \R^p \to \R$, 
	\item positive absolute error tolerance $\varepsilon_{\textrm{a}}$, and
	\item relative error tolerance $\varepsilon_{\textrm{r}} < 1$,
\end{itemize}
the  goal is to construct an \emph{optimal} approximation to $v(\mu)$, denoted  
$\hat{v}$, 
which depends on $\widehat{\bsmu}_{n}$ and $\textbf{err}_n$ and satisfies the error 
criterion
\begin{subequations} \label{FTD:eq:errtol}
\begin{gather} 
\label{FTD:eq:errtol_a}
\sup_{\bsmu \in \Omega \cap [\widehat{\bsmu}_n - \textbf{\textup{err}}_n, 
\widehat{\bsmu}_n + 
	\textbf{\textup{err}}_n] }
\textup{tol}(v(\bsmu),\hat{v},\varepsilon_{\textrm{a}},\varepsilon_{\textrm{r}}) \le 1, \\
 \textup{tol}(v,\hat{v},\varepsilon_{\textrm{a}},\varepsilon_{\textrm{r}}) := 
\frac{(v- \hat{v})^2}{\max(\varepsilon_{\textrm{a}}^2, 
	\epsilon_r^2 \abs{v}^2)}, \qquad (\varepsilon_{\textrm{a}},\varepsilon_{\textrm{r}}) \in 
	[0,\infty) \times [0,1)
	\setminus \{\bszero \}.  \label{FTD:eq:errtol_b}
\end{gather}
\end{subequations}
Our hybrid error criterion is satisfied if the actual error is no greater than either the 
absolute 
error 
tolerance or the relative error 
tolerance times the absolute value of the true answer.  If we want to satisfy both an 
absolute error criterion and a relative error 
criterion, then ``$\max$'' in the definition of $\textup{tol}(\cdot)$ should be replaced by 
``$\min$''.  This would require a somewhat different development than what is 
presented 
here.  By optimal we mean that the choice of $\hat{v}$  we prescribe yields the smallest 
possible  left hand side of \eqref{FTD:eq:errtol_a}.  This gives the greatest chance of 
satisfying the error criterion.  The dependence of $\hat{\nu}$ on 
$n$ 
is suppressed in the notation for simplicity.

The common case of estimating the integral itself, $p = 1$ and $v(\mu) = \mu$, is 
illustrated in Table \ref{FTD:table:toltable}.  This includes i) an absolute error criterion 
(see 
\eqref{FTD:eq:cubprob}), ii) a relative error criterion, and iii) a hybrid error criterion that is 
satisfied when either the absolute or relative error tolerances are satisfied.  Note that 
$\hat{v}$ is not necessarily equal to $\hat{\mu}_n$.  For a pure relative error criterion, 
$\hat{v}$ represents a shrinkage of the sample mean towards zero. Fig.\  
\ref{FTD:fig:relerr} illustrates how the optimal choice of $\hat{v}$ may satisfy 
\eqref{FTD:eq:errtol}, when $\hat{v} = \hat{\mu}$ does not.

\begin{table}
	\caption{Examples of the tolerance function in  \eqref{FTD:eq:errtol} and the optimal 
	approximation to the 
	integral when $p =1$ and $v(\mu) = \mu$.
	\label{FTD:table:toltable}}
	\[
	\begin{array}{p{1.2cm}@{\quad}c@{\quad}c@{\quad}c}
	\textbf{Kind} & \textup{tol}(\mu,\hat{v},\varepsilon_{\textrm{a}},\varepsilon_{\textrm{r}}) 
	& \text{Optimal } \hat{v} & 
	\text{Optimal } 
	\textup{tol}(\mu,\hat{v},\varepsilon_{\textrm{a}},\varepsilon_{\textrm{r}})   
	\tabularnewline
	\hline \tabularnewline [-1ex]
	\vspace{-3ex} \text{Absolute} \newline \text{$\varepsilon_{\textrm{r}} = 0$} & 
	\displaystyle \frac{(\mu-\hat{v})^2}{\varepsilon_{\textrm{a}}^2} 
	& \widehat{\mu}_n 
	& \displaystyle \frac{\textup{err}_n^2}{\varepsilon_{\textrm{a}}^2} 
	\tabularnewline [2ex]
	\vspace{-3ex} \text{Relative} \newline \text{$\varepsilon_{\textrm{a}} = 0$}
	& \displaystyle \frac{(\mu-\hat{v})^2}{\varepsilon_{\textrm{r}}^2 \mu^2} 
	& \displaystyle  \frac{\max(\widehat{\mu}_n^2 - \textup{err}_n^2, 0)}{\widehat{\mu}_n}
	& \displaystyle \frac{\textup{err}_n^2}{\varepsilon^2_r \max(\widehat{\mu}_n^2, 
	\textup{err}_n^2) } 
	\tabularnewline
	[2ex]
	\text{Hybrid} & \displaystyle \frac{(\mu-\hat{v})^2}{\max(\varepsilon_{\textrm{a}}^2, 
		\epsilon_r^2 \mu^2)}  &  \text{see \eqref{FTD:eq:vhatveqmu}}
	& \text{see \eqref{FTD:eq:tolveqmu}} 
	\end{array}
	\]
\end{table}

\begin{figure}[!ht]
	\centering
	\includegraphics[width=.75\textwidth]{\FTDFigDirectory/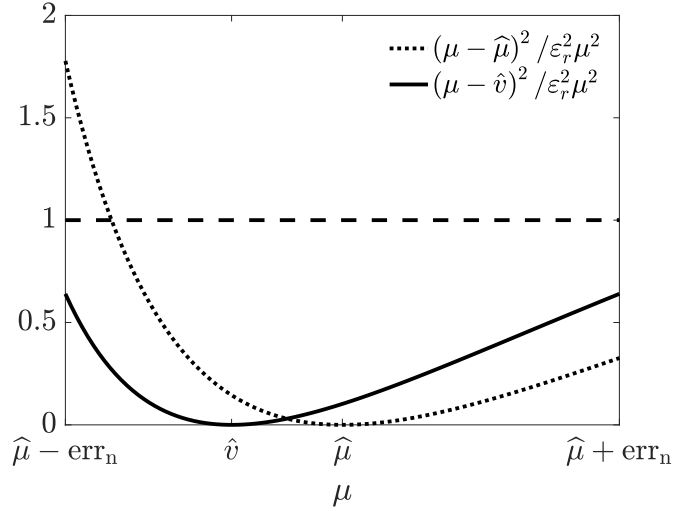}
	\caption{Example of $v(\mu) = \mu$ with the relative error criterion, i.e. 
	$\varepsilon_{\textrm{a}}=0$.  For the optimal choice of $\hat{v}$, $\sup_{\mu \in 
	[\widehat{\mu}_n - \textup{err}_n, \widehat{\mu}_n +  
	\textup{err}_n]}\textup{tol}(\mu,\hat{v},\varepsilon_{\textrm{a}},\varepsilon_{\textrm{r}})<1
	 < 
	\sup_{\mu 
	\in 
	[\widehat{\mu}_n 	- \textup{err}_n, \widehat{\mu}_n + 	
	\textup{err}_n]}\textup{tol}(\mu,\widehat{\mu}_n,\varepsilon_{\textrm{a}},\varepsilon_{\textrm{r}})$.
	\label{FTD:fig:relerr}}
\end{figure}

Define $v_{\pm}$ as the extreme values of $v(\mu)$ for $\widehat{\boldsymbol{\mu}}$ 
satisfying the given error bound:
\begin{equation} \label{FTD:eq:vpmdef}
v_- = \inf_{\bsmu \in \Omega \cap [\widehat{\bsmu}_n - \textbf{\textup{err}}_n, 
\widehat{\bsmu}_n + 
	\textbf{\textup{err}}_n] } v(\bsmu), \qquad
v_+ = \sup_{\bsmu \in \Omega \cap [\widehat{\bsmu}_n - \textbf{\textup{err}}_n, 
\widehat{\bsmu}_n + 
	\textbf{\textup{err}}_n] } v(\bsmu)
\end{equation}
Then the following criterion is equivalent to  \eqref{FTD:eq:errtol}:
	\begin{equation}  \label{FTD:eq:errtol_alt}	
	\sup_{v_- \le v' \le v_+} 
	\textup{tol}(v',\hat{v},\varepsilon_{\textrm{a}},\varepsilon_{\textrm{r}}) 
	\le 1.
	\end{equation}

We claim that the optimal value of the 
estimated integral, i.e., the value of $\hat{v}$ satisfying  \eqref{FTD:eq:errtol_alt}, is
\begin{subequations} \label{FTD:eq:optv}
	\begin{align}
	\label{FTD:eq:optv_a}
	\hat{v} &= \frac{v_-
		\max(\varepsilon_{\textrm{a}},\varepsilon_{\textrm{r}}\abs{v_+}) 
		+ v_+ 
		\max(\varepsilon_{\textrm{a}},\varepsilon_{\textrm{r}}\abs{v_-})}
	{\max(\varepsilon_{\textrm{a}},\varepsilon_{\textrm{r}}\abs{v_+}) 
		+ \max(\varepsilon_{\textrm{a}},\varepsilon_{\textrm{r}}\abs{v_-})} \\
	& = \begin{cases}
	\displaystyle \frac{v_- + v_+}{2}, &\varepsilon_{\textrm{r}} \abs{v_\pm} \le 
	\varepsilon_{\textrm{a}}, \\[1.5ex]
	\displaystyle  \frac{v_s [\varepsilon_{\textrm{a}} + v_{-s}\varepsilon_{\textrm{r}} 
	\textup{sign}(v_s)]}
	{\varepsilon_{\textrm{a}} +\varepsilon_{\textrm{r}}\abs{v_s}}, 
	&\varepsilon_{\textrm{r}} \abs{v_{-s}} \le\varepsilon_{\textrm{a}} 
	<\varepsilon_{\textrm{r}} \abs{v_{s}} , \ s 
	\in \{+,-\}, 
	\\[1.5ex]
	\displaystyle \frac{\abs{v_+v_-}[\textup{sign}(v_+) + \textup{sign}(v_-)]}
	{\abs{v_+} + \abs{v_-}}, &  
	\varepsilon_{\textrm{a}} <\varepsilon_{\textrm{r}} \abs{v_{\pm}}.
	\end{cases}
	\label{FTD:eq:threecaseopttmu}
	\end{align}
\end{subequations}
From  \eqref{FTD:eq:optv_a} it follows that $\hat{v} \in [v_-,v_+]$. Moreover, by 	
\eqref{FTD:eq:threecaseopttmu} $\hat{v}$ is a shrinkage 
estimator: it is either zero or has the same sign as 
$(v_- + v_+)/2$, and its magnitude is no greater than $\abs{(v_- + v_+)/2}$.  Our 
improved GAIL algorithms \texttt{cubSobol\_g} and \texttt{cubLattice\_g}, which are 
under development, are summarized in the following theorem.

\begin{theorem} \label{FTD:thm:gen_err}
	Let our goal be the computation of $v(\bsmu)$, as described at the beginning of this 
	section.  Let the tolerance function be defined as in 
	\eqref{FTD:eq:errtol_b}, let the extreme possible values of $v(\bsmu)$ be 
	defined as in \eqref{FTD:eq:vpmdef}, and let the approximation to $v(\bsmu)$ be 
	defined in terms of $\widehat{\bsmu}_n$ and $\textbf{\textup{err}}_n$ as in 
	\eqref{FTD:eq:optv}.  
	Then, $\hat{v}$ is the optimal approximation to 
	$v(\bsmu)$, and the tolerance function for this optimal choice is given as follows:
	\begin{subequations}\label{FTD:eq:opttolhybrid}
		\begin{align} \nonumber
	& \hspace{-8ex} \inf_{\hat{v}'} \sup_{\bsmu \in \Omega \cap [\widehat{\bsmu}_n - 
	\textbf{\textup{err}}_n, \widehat{\bsmu}_n + \textbf{\textup{err}}_n]} 
		\textup{tol}(v(\bsmu),\hat{v}',\varepsilon_{\textrm{a}},\varepsilon_{\textrm{r}}) \\
		\label{FTD:eq:opttolhybrid_z}
		& = 
		\inf_{\hat{v}'} \sup_{v_- \le v' \le v_+} 
		\textup{tol}(v',\hat{v}',\varepsilon_{\textrm{a}},\varepsilon_{\textrm{r}})  \\
		\label{FTD:eq:opttolhybrid_a}
		& = 
		\sup_{v_- \le v' \le v_+} 
		\textup{tol}(v',\hat{v},\varepsilon_{\textrm{a}},\varepsilon_{\textrm{r}}) 
		\\
		\label{FTD:eq:opttolhybrid_b} 
		& =\textup{tol}(v_\pm,\hat{v},\varepsilon_{\textrm{a}},\varepsilon_{\textrm{r}}) \\
		\label{FTD:eq:opttolhybrid_c}
		& = 
		\frac{(v_+-v_-)^2}{[\max(\varepsilon_{\textrm{a}},\varepsilon_{\textrm{r}}\abs{v_+}) 
			+ \max(\varepsilon_{\textrm{a}},\varepsilon_{\textrm{r}}\abs{v_-})]^2}.
		\end{align}
	\end{subequations}	
By optimal, we mean that the infimum in \eqref{FTD:eq:opttolhybrid_z} is satisfied by 
$\hat{v}$ as claimed in \eqref{FTD:eq:opttolhybrid_a}.  Moreover, it is shown that the 
supremum in \eqref{FTD:eq:opttolhybrid_a} is obtained simultaneously at  $v_+$ and 
$v_-$.

	Our new adaptive quasi-Monte Carlo cubature algorithms increase 
	$n = 2^m$ by incrementing $m$ by one until the right side of 
	\eqref{FTD:eq:opttolhybrid_c} is no larger than one.  The 
	resulting $\hat{v}$ then satisfies the error criterion 
	$\textup{tol}(v(\bsmu),\hat{v},\varepsilon_{\textrm{a}},\varepsilon_{\textrm{r}}) \le 1$.  
\end{theorem}

\begin{proof} 
	The gist of the proof is to establish the equalities in \eqref{FTD:eq:opttolhybrid}.  
	Equality 
\eqref{FTD:eq:opttolhybrid_c} follows from the definition of $\hat{v}$ and $v_\pm$.  
Equality \eqref{FTD:eq:opttolhybrid_b}  is proven next, and 
\eqref{FTD:eq:opttolhybrid_a} is proven after that.  Equality \eqref{FTD:eq:opttolhybrid_z} 
follows from definition \eqref{FTD:eq:vpmdef}. 

The derivative of 
$\textup{tol}(\cdot,\hat{v},\varepsilon_{\textrm{a}},\varepsilon_{\textrm{r}})$ is 
\begin{equation*}
\frac{\partial 
\textup{tol}(v',\hat{v},\varepsilon_{\textrm{a}},\varepsilon_{\textrm{r}})}{\partial 
		v'}
= \begin{cases} \displaystyle
\frac{2(v' - \hat{v})} {\varepsilon_{\textrm{a}}^2}, 
& \displaystyle \abs{v'} < \frac{\varepsilon_{\textrm{a}}}{\varepsilon_{\textrm{r}}},\\[2ex]
\displaystyle \frac{2(v' - \hat{v})\hat{v} } 
{\varepsilon_{\textrm{r}}^2v'^3},
& \displaystyle \abs{v'} > \frac{\varepsilon_{\textrm{a}}}{\varepsilon_{\textrm{r}}}.
\end{cases}
\end{equation*}
The sign of this derivative is shown in Fig.\ 	\ref{FTD:fig:signtolderiv}.
For either $\varepsilon_{\textrm{r}} \abs{v_\pm} \le 
\varepsilon_{\textrm{a}}$ or $\varepsilon_{\textrm{a}} \le\varepsilon_{\textrm{r}} 
\abs{v_\pm}$, the 
only critical point in 
$[v_-,v_+]$ is $v' = \hat{v}$, where the tolerance function vanishes.  Thus, the 
maximum value of the tolerance 
function always occurs at the boundaries of the interval.  For $\varepsilon_{\textrm{r}} 
\abs{v_{-s}} 
\le\varepsilon_{\textrm{a}} <\varepsilon_{\textrm{r}} \abs{v_{s}}$, $s \in \{+,-\} $, there is 
also a 
critical point 
at 
$v' = 
\textup{sign}(v_s)\varepsilon_{\textrm{a}}/\varepsilon_{\textrm{r}}$.  However, since 
$v_s$ and 
$\hat{v}$ have the same sign (see \eqref{FTD:eq:threecaseopttmu}), the partial 
derivative of the 
tolerance function with respect to $v'$ does not change sign at this critical point.  
Hence, 
the maximum value of the tolerance function still occurs at the boundaries of the 
interval, and  \eqref{FTD:eq:opttolhybrid_b} is established.

\begin{figure}
	\centering
	\includegraphics[width=6cm]{\FTDFigDirectory/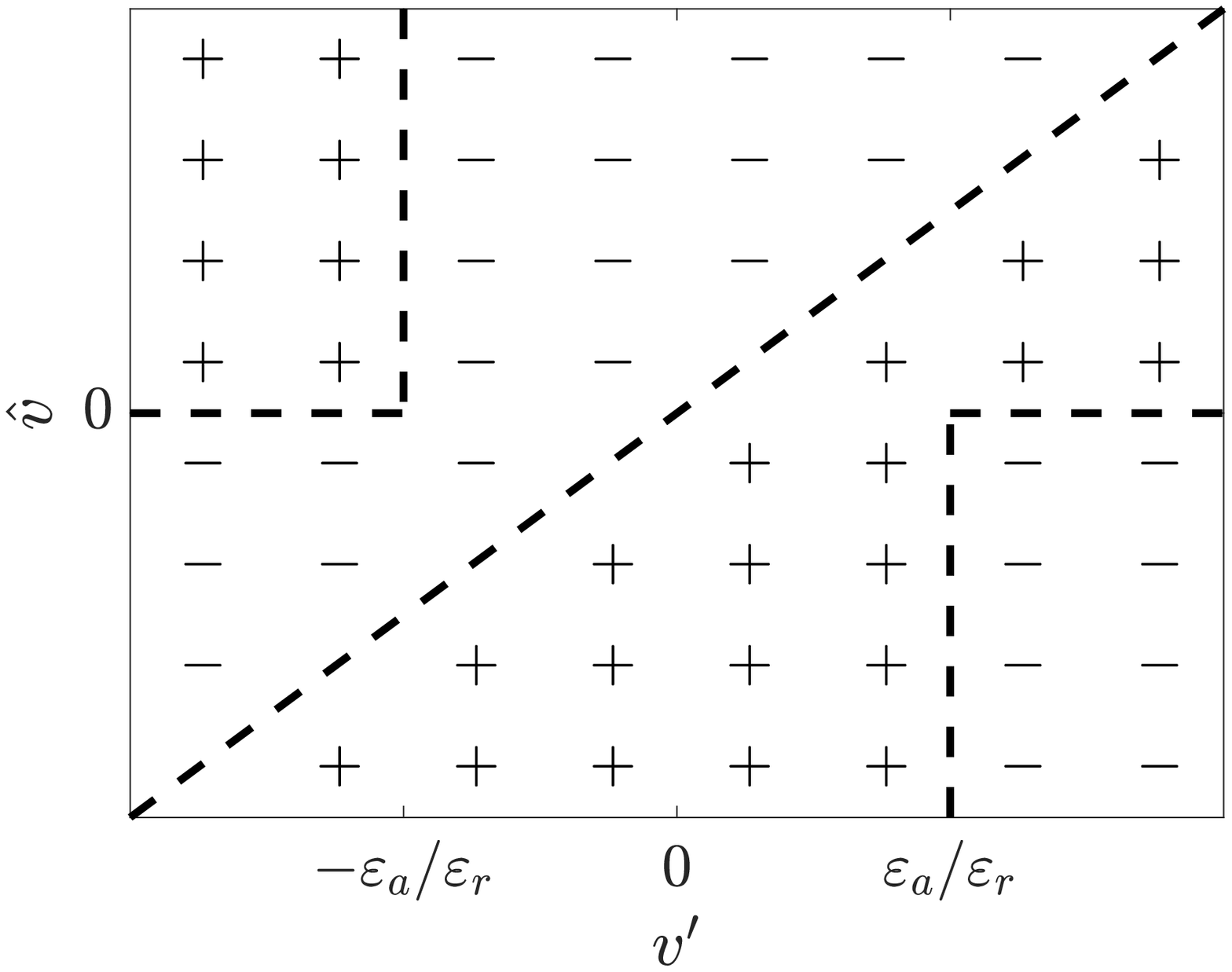}
	\caption{The sign of 
	$\partial \textup{tol}(v',\hat{v},\varepsilon_{\textrm{a}},\varepsilon_{\textrm{r}})/\partial 
	v'$. 
	\label{FTD:fig:signtolderiv}}
\end{figure}

To prove assertion \eqref{FTD:eq:opttolhybrid_a}, consider $\hat{v}'$, 
some alternative to $\hat{v}$.  Then 
\begin{align*}
\textup{tol}(v_{\pm},\hat{v}',\varepsilon_{\textrm{a}},\varepsilon_{\textrm{r}}) - 
\textup{tol}(v_{\pm},\hat{v},\varepsilon_{\textrm{a}},\varepsilon_{\textrm{r}}) 
& = \frac{(v_{\pm} - \hat{v}')^2 
	- (v_{\pm} - \hat{v})^2} 
	{\max(\varepsilon_{\textrm{a}}^2,\varepsilon_{\textrm{r}}^2v_{\pm}^2)}
	\\
& = \frac{(\hat{v}'- \hat{v} - 2v_{\pm} ) 
	(\hat{v}' - \hat{v})} 
{\max(\varepsilon_{\textrm{a}}^2,\varepsilon_{\textrm{r}}^2v_{\pm}^2)}.
\end{align*}
This difference is positive for the $+$ sign if $\hat{v}' \in (-\infty, \hat{v})$ and 
positive for  the $-$ sign if $\hat{v}' \in (\hat{v},\infty)$.  Thus, the proof of Theorem 
\ref{FTD:thm:gen_err} is complete.
\end{proof}

We return to the special case of $v(\mu) = \mu$.  The following corollary interprets 
Theorem \ref{FTD:thm:gen_err} for this case, and the theorem that follows extends 
the computational cost upper bound in \eqref{FTD:eq:abserrcost} for these new 
quasi-Monte Carlo 
cubature algorithms.
\begin{corollary} \label{FTD:cor:veqmu}
	For $p=1$ and $v(\mu) = \mu$, it follows that $v_{\pm} = \mu_n \pm \textup{err}_n$, 
	\begin{equation}
	\hat{v} = \frac{(\widehat{\mu}_n - \textup{err}_n)
		\max(\varepsilon_{\textrm{a}},\varepsilon_{\textrm{r}}\abs{\widehat{\mu}_n + 
		\textup{err}_n}) 
		+ (\widehat{\mu}_n + \textup{err}_n) 
		\max(\varepsilon_{\textrm{a}},\varepsilon_{\textrm{r}}\abs{\widehat{\mu}_n - 
		\textup{err}_n})}
	{\max(\varepsilon_{\textrm{a}},\varepsilon_{\textrm{r}}\abs{\widehat{\mu}_n + 
	\textup{err}_n}) 
		+ \max(\varepsilon_{\textrm{a}},\varepsilon_{\textrm{r}}\abs{\widehat{\mu}_n - 
		\textup{err}_n})} 
		\label{FTD:eq:vhatveqmu},
	\end{equation}
	\begin{multline}
\sup_{\widehat{\mu}_n - \textup{err}_n \le \mu \le \widehat{\mu}_n + \textup{err}_n} 
\textup{tol}(\mu,\hat{v},\varepsilon_{\textrm{a}},\varepsilon_{\textrm{r}})  \\
= 
\frac{4 
\textup{err}_n^2}{[\max(\varepsilon_{\textrm{a}},\varepsilon_{\textrm{r}}\abs{\widehat{\mu}_n
 + 
\textup{err}_n}) 
	+ \max(\varepsilon_{\textrm{a}},\varepsilon_{\textrm{r}}\abs{\widehat{\mu}_n - 
	\textup{err}_n})]^2}. 
	\label{FTD:eq:tolveqmu}
\end{multline}
	
\end{corollary}

\begin{theorem}
	For the special case described in Corollary \ref{FTD:cor:veqmu}, the computational 
	cost of obtaining an 
	approximation to the integral $\mu$ satisfying the generalized error criterion 
	$\textup{tol}(\mu,\hat{v},\varepsilon_{\textrm{a}},\varepsilon_{\textrm{r}}) \le 1$ 
	according to 
	the adaptive 
	quasi-Monte Carlo cubature algorithm 
	described in Theorem 
	\ref{FTD:thm:gen_err} is $\calO\bigl([\$(f) + m^*] 
	2^{m^*}\bigr)$, where 
	\begin{multline*}
	m^*:=\min\left\{m\geq \ell_*+r:\; \right . \\
	\left .(1+\varepsilon_{\textrm{r}})\frac{\widehat{\omega}(m) \mathring{\omega}(r)}{1 - 
	\widehat{\omega}(r) 
	\mathring{\omega}(r)}[1 + \widehat{\omega}(r) 
	\mathring{\omega}(r)]S_{m-r}(f)\leq\max(\varepsilon_{\textrm{a}},\varepsilon_{\textrm{r}}\abs{\mu})
	 \right\}.
	\end{multline*}
\end{theorem}
\begin{proof}
	For each $n = 2^m$, we know that our algorithm produces $\widehat{\mu}_n$ and 
	$\textup{err}_n$ satisfying $\widehat{\mu}_n - \textup{err}_n \le \mu \le 
	\widehat{\mu}_n + \textup{err}_n$.  
	This implies that 
	\[
	\max(\varepsilon_{\textrm{a}},\varepsilon_{\textrm{r}}\abs{\widehat{\mu}_n + 
	\textup{err}_n}) 
	+ \max(\varepsilon_{\textrm{a}},\varepsilon_{\textrm{r}}\abs{\widehat{\mu}_n - 
	\textup{err}_n}) 
	\ge 2\max(\varepsilon_{\textrm{a}},\varepsilon_{\textrm{r}}\abs{\mu}) - 
	2\varepsilon_{\textrm{r}} \textup{err}_n.
	\]	
	Thus, the right hand side of \eqref{FTD:eq:tolveqmu} must be no greater than one if 
	\[
	\textup{err}_n \le 
	\frac{\max(\varepsilon_{\textrm{a}},\varepsilon_{\textrm{r}}\abs{\mu})}{1+\varepsilon_{\textrm{r}}}.
	\]
	Applying the logic that leads to \eqref{FTD:eq:abserrcost} completes the proof.
\end{proof}

The cost upper bound depends on various parameters as one would expect.  The 
computational cost may increase if 
\begin{itemize}
	\item $\varepsilon_{\textrm{a}}$ decreases,
	\item $\varepsilon_{\textrm{r}}$ decreases,
	\item $\abs{\mu}$ decreases,
	\item the Fourier coefficients of the integrand increase, or
	\item the cone $\mathcal{C}$ expands because $\ell_*$, $\widehat{\omega}$, or 
	$\mathring{\omega}$ increase.
\end{itemize}

\section{Numerical Implementation} \label{FTD:sec:numerical}

The algorithm described here is intended to be released in the next release of GAIL 
\cite{ChoEtal15a} as \texttt{cubSobol\_g} and \texttt{cubLattice\_g}, coded 
in MATLAB.  These two functions use the Sobol' sequences provided by 
MATLAB  2017a \cite{MAT9.2} 
and the lattice generator
\texttt{exod2\_base2\_m20.txt} from  Dirk Nuyens'  website \cite{Nuy17a}, respectively.  
Our algorithm sets its default parameters as follows:
\begin{equation}
\ell_* = 6, \qquad r = 4, \qquad \frakC(m) = 5 \times 2^{-m}.
\end{equation}
These choices are based on experience and are used in the examples below.  A larger 
$\ell_*$ allows the Fourier 
coefficients of the integrand to behave erratically over a larger initial segment of 
wavenumbers.  A larger $r$ decreases the impact of aliasing in estimating the true 
Fourier coefficients by their discrete analogues.  Increasing $\ell_*$ or $r$ increases 
$2^{\ell_*+r}$, the minimum number of sample points used by the algorithms.  The inputs 
to the algorithms are 
\begin{itemize}
	\item a black-box $p$-vector function $\bsf$, such that $\bsmu = 
	\mathbb{E}[\bsf(\bsX)]$ for $\bsX \sim \calU[0,1]^d$, 
	\item a solution function $v:\R^p \to \R$,
	\item functions for computing $v_{\pm}$ as described in \eqref{FTD:eq:vpmdef},
	\item an absolute error tolerance, $\varepsilon_{\textrm{a}}$, and
	\item a relative error tolerance $\varepsilon_{\textrm{r}}$.
\end{itemize}
The algorithm increases $m$ incrementally until the right side of 
\eqref{FTD:eq:opttolhybrid_c} does 
not exceed one.  At this point the algorithm returns $\hat{v}$ as given by 
\eqref{FTD:eq:optv}.

\begin{example}\label{FTD:ex:exampleMulti}
We illustrate the hybrid error criterion by estimating multivariate normal probabilities for a 
distribution with mean $\bszero$ and covariance matrix $\mathsf{\Sigma}$:
\begin{equation}\label{FTD:eq:multivariateProba}
v(\mu) = \mu = \mathbb{P}\left[\bsa \leq \bsX\leq \bsb\right]=\int_{[\bsa, \bsb]} 
\frac{\E^{-\bsx^T\mathsf{\Sigma}^{-1}\bsx/2}}{(2\pi)^{d/2}\abs{\mathsf{\Sigma}}^{1/2}}\,\D\bsx.
\end{equation}
The transformation proposed by Genz \cite{Gen93} is used write this as an 
integral over the $d-1$ dimensional unit cube.  As discussed in \cite{Gen93,HicHon97a}, 
when $\bsa=-\boldsymbol{\infty}$, 
$\mathsf{\Sigma}_{ij}=\sigma$ if $i\neq j$, and $\mathsf{\Sigma}_{ii}=1$, the exact value 
of \eqref{FTD:eq:multivariateProba} reduces to  a 1-dimensional 
integral that can be accurately estimated by a standard quadrature rule. This value is 
taken to be the true $\mu$.

We perform $1000$ adaptive integrations:  $500$ 
using our cubature rule based on randomly scrambled and digitally shifted Sobol' 
sequences (\texttt{cubSobol\_g}) and $500$ using our cubature rule based on 
randomly shifted rank-1 lattice node sequences, (\texttt{cubLattice\_g}).  Default 
parameters are used.  For each case we 
choose $\sigma \sim \mathcal{U}[0,1]$, dimension $d = \lfloor 
500^D\rfloor$ with $D\sim\mathcal{U}[0,1]$, and $\bsb \sim 
\mathcal{U}[0,\sqrt{d}]^d$. The dependence of $\bsb$ on the dimension of the problem 
ensures that the estimated probabilities are of the same order of magnitude for all $d$. 
Otherwise, the higher the dimension, the smaller the value of the probabilities the test 
would be estimating. The execution time and $\textup{tol}(\mu,\hat{v},0.01,0.05)$ 
are 
shown in Fig.\ \ref{FTD:fig:testHybridErrorMulti}.

\begin{figure}[!ht]
\centering
\includegraphics[width=.47\textwidth]{\FTDFigDirectory/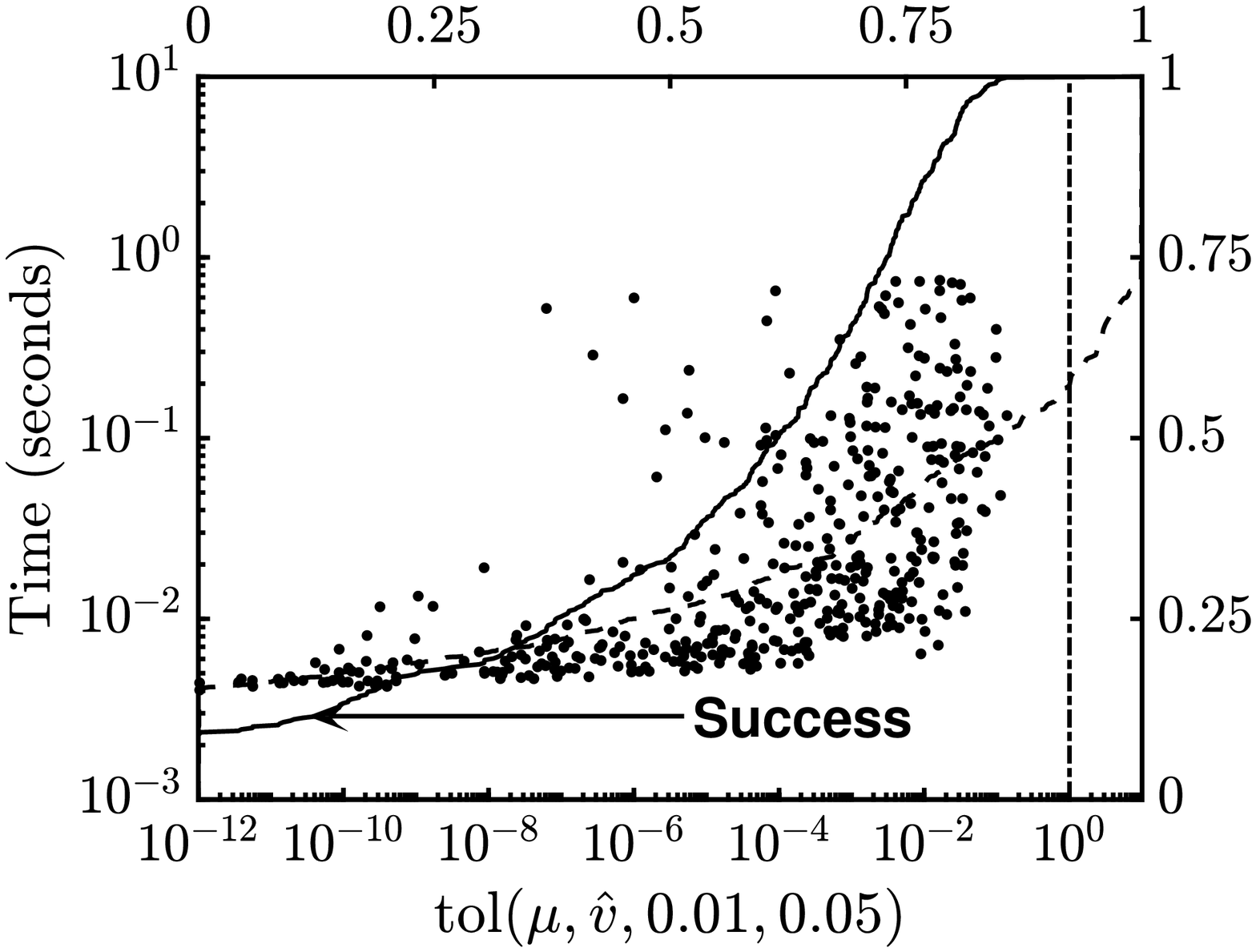}\quad
\includegraphics[width=.47\textwidth]{\FTDFigDirectory/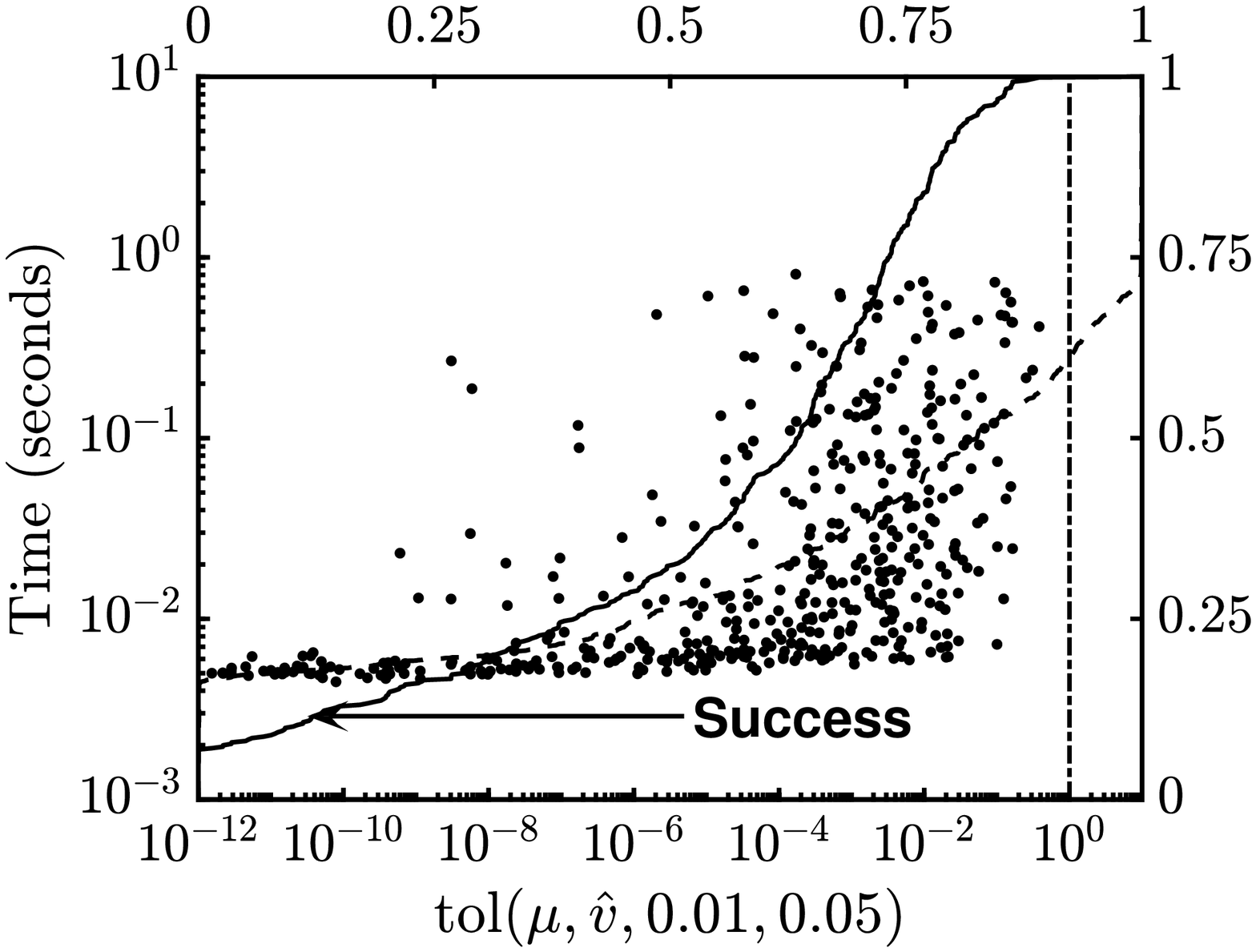}
\caption{On the left, $500$ integration results using scrambled and 
digitally shifted Sobol' sequences, \texttt{cubSobol\_g}. On the right, tolerance values 
and computation times of integrating $500$ 
multivariate normal probabilities using randomly shifted rank-1 lattice node sequences, 
\texttt{cubLattice\_g}.  If an integrand is in $\mathcal{C}$, its dot 
must lie to  the left of the vertical dot-dashed line denoting 
$\textup{tol}(\mu,\hat{v},0.01,0.05)=1$. 
The solid and dashed curves represent the 
empirical distributions of tolerance values and times 
respectively.}\label{FTD:fig:testHybridErrorMulti}
\end{figure}

Satisfying the error criterion is equivalent to having
$\textup{tol}(\mu,\hat{v},0.01,0.05) \le 1$, which happens in every case.  A very small 
value of 
$\textup{tol}(\mu,\hat{v},0.01,0.05)$ means that the approximation is much more 
accurate than 
required, which may be due to coincidence or due to the minimum sample size used, $n 
= 2^{10}$.  In Fig.\ \ref{FTD:fig:testHybridErrorMulti} the error tolerances are fixed and 
do not affect the computation time.  However, the computation time does depend on the 
dimension, $d$, since higher dimensional problems tend to be 
harder to solve. The performances of 
\texttt{cubSobol\_g} and \texttt{cubLattice\_g} are similar.

\end{example}

\begin{example}\label{FTD:ex:SobolIndEx}  Sobol' indices \cite{Sob90,Sob01}, which 
arise in uncertainty quantification, depend on more than 
one 
integral.  Suppose that one is interested in how an output, $Y := g(\bsX)$  depends on 
the 
input $\bsX\sim 
\mathcal{U}[0,1]^d$, and $g$ has a complicated or 
unknown structure.  For example, 
$g$ might be the output of a computer simulation.  
For any coordinate indexed by $j=1,\dots,d$,  the normalized closed first-order Sobol' 
index for coordinate $j$, commonly denoted as $\underline{\tau}_j^2/\sigma^2$, involves 
three integrals:
\begin{subequations} \label{FTD:eq:sob_ind}
\begin{gather}
v(\bsmu) := \frac{\mu_1}{\mu_2-\mu_3^2}, \qquad
\mu_1: = \int_{[0,1)^{2d}} [g(\bsx_j:{\bsx'}_{-j})-
	g(\bsx')] g(\bsx) \, \D\bsx \, \D\bsx', \\
	\mu_2 : = \int_{[0,1)^{d}} g(\bsx)^2 \, \D\bsx, \qquad
	\mu_3 := \int_{[0,1)^{d}} g(\bsx) \, \D\bsx.
\end{gather}
\end{subequations}
Here, $(\bsx_j:{\bsx'}_{-j})\in [0,1)^d$ denotes a point whose $r^{\text{th}}$ coordinate is 
$x_r$ if $r=j$, and  $x'_r$ otherwise.  By definition, the value of these normalized indices 
must lie between $0$ and $1$, and both the numerator and denominator in the 
expression for $v(\bsmu)$ are non-negative. Therefore, the domain of the function $v$ 
is 
$\Omega := \{\bsmu \in [0,\infty)^3 :  0 \le \mu_1 \le \mu_2 - \mu_3^2 \}$.   Thus, 
given 
$\widehat{\bsmu}_n$ and $\textbf{\textup{err}}_n$, the 
values of $v_{\pm}$ defined in  \eqref{FTD:eq:vpmdef} are 
\begin{equation}
v_\pm = \begin{cases} 
0, \qquad \mu_{n,1} \pm \textup{err}_{n,1}\leq 0, \\
1, \qquad \mu_{n,1} \pm \textup{err}_{n,1} > \max\bigl(0, \mu_{n,2} \mp \textup{err}_{n,2} 
-(\mu_{n,3} \pm 
\textup{err}_{n,3})^2\bigr), \\
\displaystyle \frac{\mu_{n,1} \pm \textup{err}_{n,1}}{\mu_{n,2} \mp \textup{err}_{n,2} 
-(\mu_{n,3} \pm 
	\textup{err}_{n,3})^2}, \qquad \text{otherwise}.
\end{cases}
\end{equation}

We estimate the first-order Sobol' indices of the test function in Bratley \emph{et al.} 
\cite{BraFoxNie92} using randomly scrambled and digitally shifted Sobol' sequences and 
the same algorithm parameters as in Ex.\ 
\ref{FTD:ex:exampleMulti}:
\[
g(\bsX)=\sum_{i=1}^6 (-1)^i \prod_{j=1}^i 
\bsX_j.
\]
\begin{equation*}
\begin{array}{ r | ccccccc }
	j& 1  &  2 & 3 &  4 & 5	& 6 \\
	n & 8~192 & 4~096 & 1~024 & 1~024 & 1~024 & 1~024  \\
	v & 0.6529 & 0.1791 & 0.0370 & 0.0133 & 0.0015 & 0.0015\\
	\hat{v} & 0.6495 &  0.1794  &  0.0336  &  0.0134  &  0.0010 & 0.0016\\
	v(\widehat{\bsmu}_n) & 0.6402 & 0.1700 & 0.0336 & 0.0134 & 0.0009 & 0.0016\\
    \textup{tol}(v,\hat{v},5 \times10^{-3},0) 
    & 0.4469 & 0.0024 & 0.4646 & 0.0003 & 0.0112 &  0.0002 \\
	\textup{tol}(v,v(\widehat{\bsmu}_n),5 \times10^{-3},0) 
	& 6.4623 & 3.3313 & 0.4765 & 0.0002 & 0.0113 & 0.0002 
\end{array}
\end{equation*}
The value of $n$ chosen by our adaptive algorithm and the actual value of the tolerance 
function,
$\textup{tol}(v,\hat{v},5 \times 10^{-3},0)$, are shown.  Since none of those tolerance 
values 
exceed one, our algorithm correctly provides $\hat{v}$ for each coordinate $j$. In 
the last row above, we replaced our optimal $\hat{v}$ by $v(\widehat{\bsmu}_n)$ for the 
same $n$ as returned by our algorithm. Interestingly, this approximation to the Sobol' 
indices, while perhaps intuitive, does not satisfy the absolute error criterion because 
sometimes $\textup{tol}(v,v(\widehat{\bsmu}_n),5 \times 10^{-3},0)$ exceeds one.  This 
reflects 
how  $v(\widehat{\bsmu}_n)$ differs from $v$ much more than $\hat{v}$ 
does.  An extensive study on how to estimate first-order and total effect Sobol' indices 
using the automatic quasi-Monte Carlo cubature is provided in \cite{GilJim16b}.
\end{example}

\section{Control Variates} \label{FTD:sec:CV}
The results in this section mainly follow the work of Da Li \cite{Li16a}. Control variates 
are commonly used to improve the efficiency of IID Monte Carlo 
integration.  If one chooses a vector of functions $\bsg:[0,1)^d \to \R^q$ for which 
$\bsmu_\bsg := 
\int_{[0,1)^d} \bsg(\bsx) \, \D \bsx$ is known, then 
\begin{equation*}
\mu := \int_{[0,1)^d} f(\bsx) \, \D \bsx = \int_{[0,1)^d} h_{\bsbeta}(\bsx) \, \D \bsx, \qquad 
\text{where } h_{\bsbeta}(\bsx):=f(\bsx)+\bsbeta^T(\bsmu_\bsg-\bsg(\bsx)),
\end{equation*}
for any choice of $\bsbeta$.  The goal is to choose an optimal $\bsbeta$ to make 
\begin{equation*}
\widehat{\mu}_{\bsbeta,n} : = \frac 1n \sum_{i=0}^{n-1} h_{\bsbeta}(\bsx_i)
\end{equation*}
sufficiently close to  $\mu$ with the least expense, $n$, possible.

If $\bsx_0, \bsx_1, \ldots$ are  IID $\calU[0,1)^d$, then 
$\widehat{\mu}_{\bsbeta,n}$ is an unbiased estimator for $\mu$ for any choice of $\bsbeta$, and the variance of the control variates estimator  may be expressed as 
\begin{equation*}
\textup{var}(\widehat{\mu}_{\bsbeta,n}) = \frac{\textup{var}(h_{\bsbeta}(\bsx_0))}{n} = 
\frac 1n 
\sum_{\kappa=1}^\infty \abs{\hat{f}_{\kappa}-\bsbeta^T\hat{\bsg}_{\kappa}}^2,
\end{equation*}
where $\hat{\bsg}_{\kappa}$ are the Fourier coefficients of $\bsg$.  Since 
$\bsbeta^T\bsmu_\bsg$ is constant, it does not enter into the calculation of the 
variance.  The optimal choice 
of $\bsbeta$, which minimizes $\textup{var}(\widehat{\mu}_{\bsbeta,n})$, is 
\[
\bsbeta_{\textup{MC}}=\frac{\textup{cov}\bigl(f(\bsx_0),\bsg(\bsx_0)\bigr)}{\textup{var}\bigl
(\bsg(\bsx_0) 
\bigr)}.
\]
Although $\bsbeta_{\textup{MC}}$ cannot be computed exactly, it may be well 
approximated in terms of sample estimates of the quantities on the right hand side. 

However, if $\bsx_0, \bsx_1, \ldots$ are the points described in Sec. \ref{FTD:sec:ErrEst}, then the error depends on only some of the Fourier coefficients, and \eqref{FTD:eq:error_as_Fourier_kappa} and \eqref{FTD:eq:data_err_bd} lead to
\begin{multline} \label{FTD:eq:CV_data_err_bd}
\abs{\mu - \widehat{\mu}_{\bsbeta,n}} \le \sum_{\lambda = 1}^{\infty} 
\abs{\hat{f}_{\lambda 
2^m} - \bsbeta^T\hat{\bsg}_{\lambda 
2^m} } \le  \frac{\widehat{\omega}(m) \mathring{\omega}(r)}{1 - 
	\widehat{\omega}(r) 
	\mathring{\omega}(r)}\widetilde{S}_{m-r,m}(f - \bsbeta^T\bsg),
\\
\text{provided } f - \bsbeta^T\bsg \in \calC.
\end{multline}
Assuming that $f - \bsbeta^T\bsg \in \calC$ for all $\bsbeta$, it makes 
sense to choose $\bsbeta$ to minimize the rightmost term.  There seems to be some 
advantage to choose 
$\bsbeta$ based on $\widetilde{S}_{m-r,m}(f - \bsbeta^T\bsg), \ldots, 
\widetilde{S}_{m,m}(f - \bsbeta^T\bsg)$.  Our experience 
suggests that this strategy makes $\bsbeta$ less dependent on the fluctuations of the 
discrete  
Fourier coefficients over a small range of wave numbers.  In summary,
\begin{equation*}
\bsbeta_{\textup{qMC}}= \argmin_\bsb \sum_{t=0}^r \widetilde{S}_{m-t,m}(f - \bsb^T\bsg) 
= \argmin_\bsb  
\sum_{\kappa=\left \lfloor 2^{m-r-1}\right \rfloor}^{2^{m}-1} 
\abs{\tilde{f}_{m,\kappa} - \bsb^T\tilde{\bsg}_{m,\kappa}}.
\end{equation*}
As already noted in \cite{HicEtal03}, the optimal control variate coefficients for IID and 
low discrepancy sampling are generally different. Whereas $\bsbeta_{\textup{MC}}$ may 
be strongly influenced by low wavenumber Fourier coefficients of the integrand, 
$\bsbeta_{\textup{qMC}}$ depends on rather high wavenumber Fourier coefficients.

Minimizing the sum of absolute values is computationally more time consuming than 
minimizing the sum of squares.  Thus, in practice we choose $\bsbeta$ to be
 \begin{equation*}
\widetilde{ \bsbeta}_{\textup{qMC}}= \argmin_\bsb  
 \sum_{\kappa=\left \lfloor 2^{m-r-1}\right \rfloor}^{2^{m}-1} \abs{\tilde{f}_{m,\kappa} - 
 \bsb^T\tilde{\bsg}_{m,\kappa}}^2.
 \end{equation*}
This choice performs well in practice.  Moreover, we often find that there is little 
 advantage to updating $\widetilde{ \bsbeta}_{\textup{qMC}}$ for each $m$.

\begin{example}
Control variates may be used to expedite the pricing of an exotic option when one 
can identify a similar option whose price is known exactly.  This often happens with 
geometric Brownian motion asset price models.  The \emph{geometric} mean Asian 
payoff is a good control variate 
for estimating the price of an \emph{arithmetic} mean Asian option. The two payoffs are,
\begin{align*}
f(\bsx) & = \E^{-r T}\max\left(\frac{1}{d}\sum_{j=1}^d S_{t_j}(\bsx)-K, 0\right) = 
\text{arithmetic mean Asian call}, \\
g(\bsx) & = \E^{-r T}\max\left(\left[\prod_{j=1}^d S_{t_j}(\bsx)\right]^{1/d}-K, 0\right)= 
\text{geometric mean Asian call}, \\
S_{t_j}(\bsx) & =  S_0\E^{(r-\sigma^2/2)t_j+\sigma Z_j(\bsx)} = \text{stock price at time } 
t_j, \\
\begin{pmatrix}
Z_1(\bsx) \\ \vdots \\ Z_d(\bsx)
\end{pmatrix}
& = \mathsf{A} 
\begin{pmatrix}
\Phi^{-1}(x_1) \\ \vdots \\ \Phi^{-1}(x_d)
\end{pmatrix}, \qquad \mathsf{A} \mathsf{A}^T = \mathsf{C} := \Bigl( 
\min(t_i,t_j)\Bigr)_{i,j=1}^d.
\end{align*}
Here $\mathsf{C}$ is the covariance matrix of the values of a Brownian motion at the discrete times $t_1, \ldots, t_d$.   We choose $\mathsf{A}$ via a principal component 
analysis (singular value) decomposition of $\mathsf{C}$ as this tends to provide quicker convergence to the answer than other choices of $\mathsf{A}$.  

The option parameters for this example are $S_0=100$, $r=2\%$, $\sigma=50\%$, 
$K=100$, and $T=1$. We employ weekly monitoring, so $d=52$, and   
$t_j=j/52$, where the option price is about $\$11.97$. Parameter $\tilde{ \bsbeta}_{\textup{qMC}}$ is 
estimated at the 
first iteration of the algorithm when $m=10$, but not updated for each $m$. For $\varepsilon_{\textrm{a}}= 0.01$ and 
$\varepsilon_{\textrm{r}}=0$, \texttt{cubSobol\_g} without control variates requires 
$16~384$ points while only $4~096$ when using control variates.

Fig.\ \ref{FTD:fig:testAMeanCV} shows the Fourier Walsh coefficients of the original 
payoff, $f$, and the function integrated using control variates, $h_{\tilde{ 
\beta}_{\textup{qMC}}} = f + 
\tilde{ \beta}_{\textup{qMC}}(\mu_g - g)$, with given 
$\tilde{ \beta}_{\textup{qMC}}=1.0793$, a typical value of $\beta$ chosen by our 
algorithm. The squares correspond to the coefficients 
in the sums $\widetilde{S}_{6,10}(f)$ and $\widetilde{S}_{6,10}(h_{\tilde{ 
\beta}_{\textup{qMC}}})$, respectively, which are used to bound the Sobol' cubature 
error.  The circles 
are the first coefficients from the dual net that appear in error bound 
\eqref{FTD:eq:error_as_Fourier_kappa}. From this Fig.\ we can appreciate how control 
variates reduces the magnitude of both the  squares and the  circles.
\begin{figure}[!ht]
\centering
\includegraphics[width=.45\textwidth]{\FTDFigDirectory/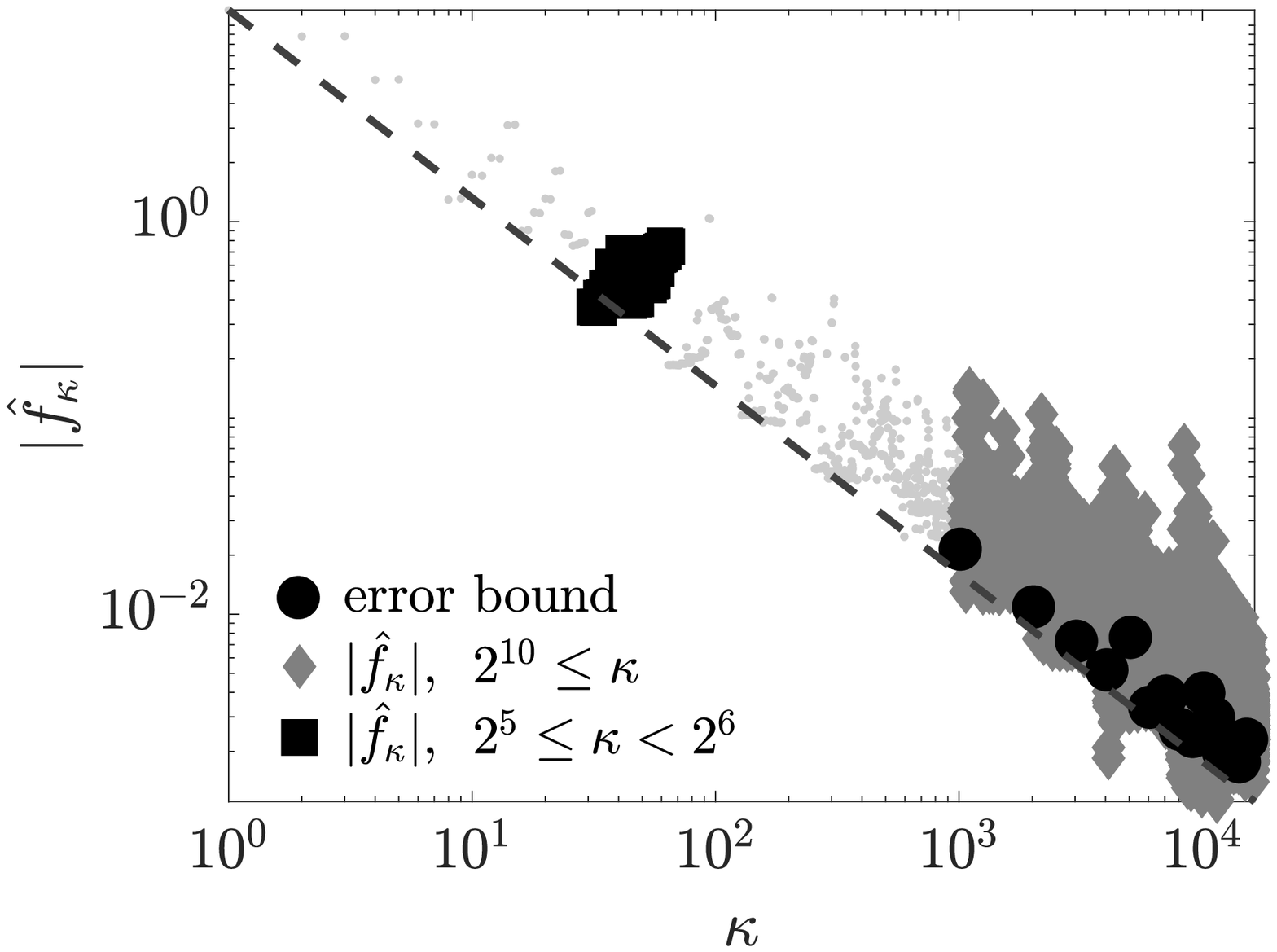} \qquad
\includegraphics[width=.45\textwidth]{\FTDFigDirectory/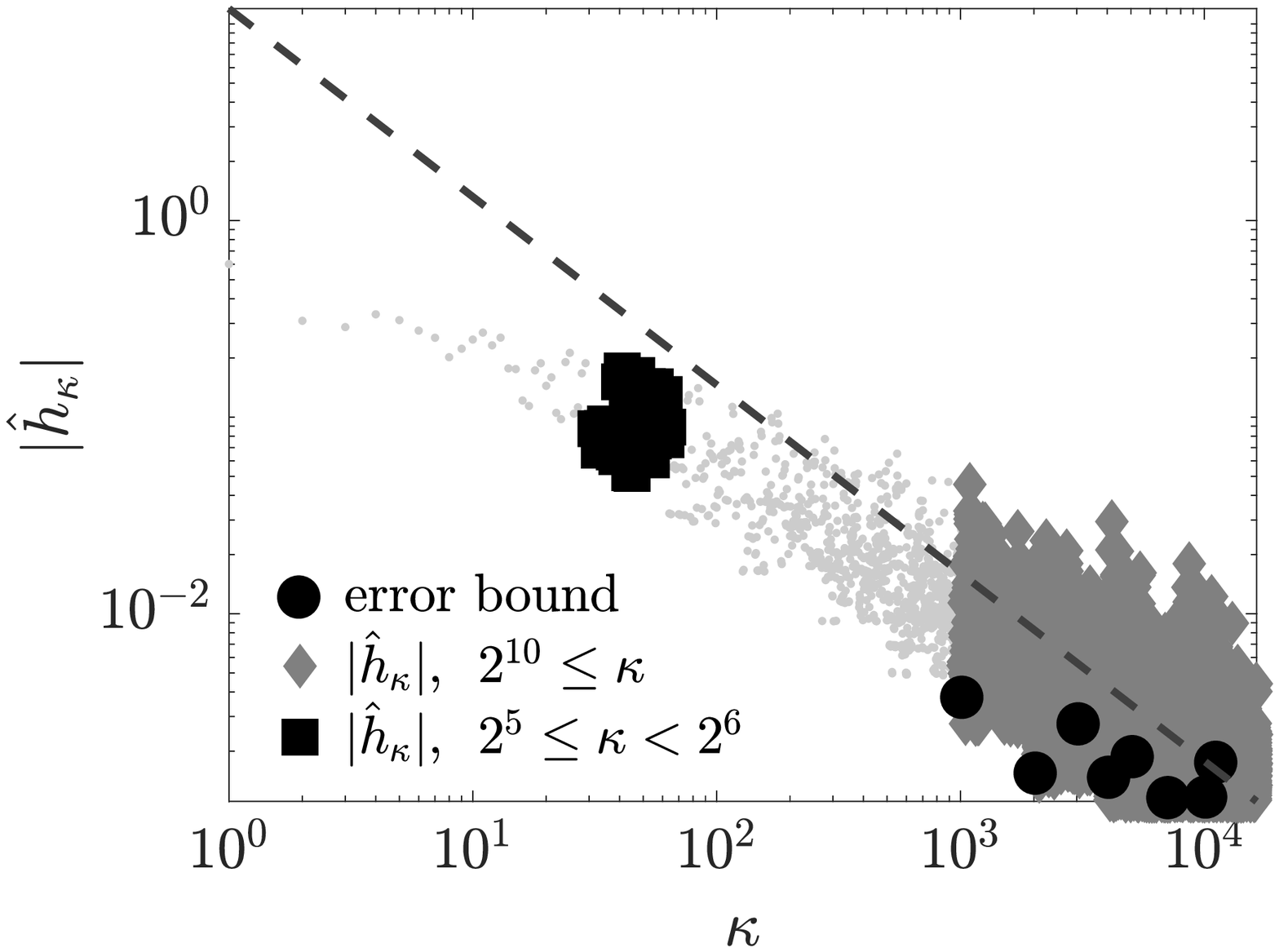}
\caption{Fourier Walsh coefficients for $f(\bsx)$ on the left and $h_{1.0793}(\bsx)$ on 
the right. The value $\tilde{\bsbeta}_{\textup{qMC}}$ effectively decreased the size of 
the coefficients involved in either the data-based error bound \eqref{FTD:eq:data_err_bd} 
and the error bound \eqref{FTD:eq:error_as_Fourier_kappa}.}\label{FTD:fig:testAMeanCV}
\end{figure}
\end{example}

\section{Discussion and Conclusion}
Ian Sloan has made substantial contributions to the understanding and practical 
application  of quasi-Monte Carlo cubature.
One challenge is how to choose the parameters that define these cubatures in 
commonly encountered situations where not much is known about the integrand.  These 
parameters include
\begin{enumerate}
	\renewcommand{\labelenumi}{\alph{enumi})}
	\item  the generators of the sequences themselves,
	\item  the sample size, $n$,
	\item the choice of importance sampling distributions, 
	\item the control variate coefficients \cite{HicEtal03}, 
	\item the parameters defining multilevel (quasi-)Monte Carlo methods \cite{Gil15a}, 
	and the 
	\item the parameters defining the multivariate decomposition method \cite{Was13b}.
\end{enumerate}
The rules for choosing these parameters should work well in practice, but not 
be simply heuristic as they are in the adaptive algorithms highlighted in the introduction.  
There should be a theoretical justification.  Item a) has  received 
much attention.  This article has addressed items b) and d).  We realize that the question 
of choosing $n$ is now replaced by the question of choosing the parameters defining the 
cone of integrands, $\calC$.  However, we have made
progress because when 
our adaptive algorithms fail, we can pinpoint the cause.  We hope for further 
investigations into the best way to choose $n$.  We also hope that further efforts 
will lead to more satisfying answers for the other items on the list.

As demonstrated in Sec.\ \ref{FTD:sec:GeneralError}, it is now possible to set relative 
error criteria or hybrid error criteria.  We also know now how to accurately estimate a 
function of several means.  In addition to the problem of Sobol' indices, this problem may 
arise in Bayesian inference, where the posterior mean of a parameter is the quotient of 
two integrals.
 
As already pointed out some years ago in \cite{HicEtal03}, the choice of control variate 
for IID sampling is not necessarily the right choice for low discrepancy sampling.  Here in 
Sec.\ \ref{FTD:sec:CV}, we have identified a natural way to determine a good control 
variate coefficient for digital sequence or lattice sequence sampling.

\bibliographystyle{spmpsci}
\bibliography{FJH23,FJHown23,thesis_lluisantoni}

\end{document}

%% file: macros.tex
\renewcommand{\email}[1]{\emailname: #1} 

\usepackage{mathptmx}       
\usepackage{helvet}         
\usepackage{courier}        

\usepackage{makeidx}         
\usepackage{graphicx}        
\usepackage[bottom]{footmisc}

\usepackage{latexsym}
\usepackage{amsmath}
\usepackage{amsfonts}
\usepackage{amssymb}
\usepackage{bm}

\usepackage{url}
\usepackage{algorithm}
\usepackage{algorithmic}
\usepackage[misc,geometry]{ifsym}

\renewenvironment{proof}{\noindent{\itshape Proof.}}{\smartqed\qed}

\spdefaulttheorem{assumption}{Assumption}{\upshape \bfseries}{\itshape}
\spdefaulttheorem{algo}{Algorithm}{\upshape \bfseries}{\itshape}

\newcommand{\bsa}{{\boldsymbol{a}}}
\newcommand{\bsb}{{\boldsymbol{b}}}

\newcommand{\bsf}{{\boldsymbol{f}}}
\newcommand{\bsg}{{\boldsymbol{g}}}

\newcommand{\bsk}{{\boldsymbol{k}}}
\newcommand{\bsl}{{\boldsymbol{l}}}

\newcommand{\bsx}{{\boldsymbol{x}}}

\newcommand{\bsz}{{\boldsymbol{z}}}

\newcommand{\bsX}{{\boldsymbol{X}}}

\newcommand{\bszero}{{\boldsymbol{0}}} 
\newcommand{\bsone}{{\boldsymbol{1}}}  

\newcommand{\bsbeta}{{\boldsymbol{\beta}}}

\newcommand{\bsmu}{{\boldsymbol{\mu}}}

\newcommand{\bsDelta}{{\boldsymbol{\Delta}}}

\newcommand{\bbE}{{\mathbb{E}}}

\newcommand{\bbK}{{\mathbb{K}}}

\newcommand{\bbZ}{{\mathbb{Z}}}
\newcommand{\N}{{\mathbb{N}}} 
\newcommand{\R}{{\mathbb{R}}} 

\DeclareSymbolFont{bbold}{U}{bbold}{m}{n}
\DeclareSymbolFontAlphabet{\mathbbold}{bbold}

\newcommand{\calB}{{\mathcal{B}}}
\newcommand{\calC}{{\mathcal{C}}}

\newcommand{\calO}{{\mathcal{O}}}

\newcommand{\calU}{{\mathcal{U}}}

\newcommand{\frakC}{{\mathfrak{C}}}

\providecommand{\argmin}{\operatorname*{argmin}}